\documentclass[a4paper,11pt]{amsart}

\usepackage[utf8]{inputenc}
\usepackage[english]{babel}
\usepackage[pdftex]{graphicx}
\usepackage{amsmath,amsfonts}
\usepackage{color}
\usepackage[left=2.5cm,top=4cm,right=2.5cm,bottom=4cm]{geometry}
\usepackage[hidelinks]{hyperref}
\usepackage{fancyhdr}
\usepackage[font={small,it}]{caption}
\usepackage{subcaption}
\usepackage{mathtools}
\usepackage{caption}
\usepackage{relsize}
\usepackage{amssymb}
\usepackage{amsthm}
\usepackage{mathrsfs}
\usepackage[parfill]{parskip}
\usepackage{enumitem}   
\usepackage{algorithm}
\usepackage{algpseudocode}
\usepackage{fancyhdr}
\usepackage[dvipsnames]{xcolor}
\usepackage{tcolorbox}
\usepackage{xspace}
\usepackage{csquotes}

\usepackage{mathtools}
\mathtoolsset{showonlyrefs=true}

\usepackage[backend=biber,style=numeric,giveninits,maxnames=1000]{biblatex}
\usepackage{bbm}
\date{\today}

\numberwithin{equation}{section}

\newtheorem{theorem}{Theorem}[section]
\newtheorem{corollary}[theorem]{Corollary}
\newtheorem{lemma}[theorem]{Lemma}
\newtheorem{proposition}[theorem]{Proposition}
\theoremstyle{definition}
\newtheorem{definition}[theorem]{Definition}

\newtheorem{remark}[theorem]{Remark}
\newtheorem{assumption}[theorem]{Assumption}

\graphicspath{{Fig/}}

\numberwithin{equation}{section}

\usepackage{xspace}
\usepackage{mathtools}
\mathtoolsset{showonlyrefs=true}

\newcommand \R {\mathbb{R}}
\renewcommand \P {\mathbb{P}}
\newcommand \F {\mathcal{F}}
\newcommand \E {\mathbb{E}}
\newcommand \N {\mathbb{N}}
\newcommand \calS {\mathcal{S}}
\newcommand \J {\mathcal{J}}
\newcommand \calH {\mathcal{H}}
\newcommand \T {\mathcal{T}}

\newcommand \dm {{\diamond}}
\newcommand \Dt {\Delta t}
\newcommand \M {{M_H}} 

\newcommand{\tZ}{{\tilde Z}}
\newcommand{\tJ}{{\tilde J}}
\newcommand{\tZt}{{\tZ^t}}
\newcommand{\tJt}{{\tJ^t}}
\newcommand{\tZT}{{\tZ^T}}
\newcommand{\tJT}{{\tJ^T}}

\newcommand{\br}[1]{\left(#1\right)} 
\newcommand{\fbm}[1]{B^H({#1})} 
\newcommand{\der}[2]{\frac{\partial #1} {\partial #2} } 
\newcommand{\ind}[1]{\mathbf{1}_{[0,#1]}} 
\newcommand{\wexp}[1]{\exp^\dm\left(#1\right)} 
\newcommand \intsum {\sum_{i=0}^{n-1}\int_{t_i}^{t_{i+1}}} 
\newcommand{\expect}[1]{{\E\left[#1\right]}} 

\newcommand{\GBMEM}{\texttt{GfBm\_EM}\xspace}
\newcommand{\MISHURA}{\texttt{Mishura\_EM}\xspace}
\newcommand{\EXPFREEZE}{\texttt{ExpFreeze}\xspace}
\newcommand{\Rosenbrock}{\texttt{Rosenbrock}\xspace}

\begin{document}

\title[Numerical Approximation of WIS-SDEs]{Numerical approximation of SDEs driven by fractional Brownian motion for all $H\in(0,1)$ using WIS integration.}

\author{U. Erdo\v gan} 
\address{Department Of Mathematics, Eskisehir Technical University, Eskisehir, Turkiye}
\curraddr{}
\email{utkuerdogan@eskisehir.edu.tr}

\author{G.~J. Lord}
\address{Mathematics, IMAPP, Radboud University, Nijmegen, The Netherlands.}
\email{gabriel.lord@ru.nl}
\thanks{}

\author{R.~B. Schieven}
\address{Korteweg--De Vries Institute, University of Amsterdam, Amsterdam, The Netherlands.}
\curraddr{}
\email{r.b.schieven@uva.nl}
\thanks{}

\begin{abstract} 
We examine the numerical approximation of a quasilinear stochastic differential equation (SDE) with multiplicative fractional Brownian motion. The stochastic integral is interpreted in the Wick-It\^o-Skorohod (WIS) sense that is well defined and centered for all $H\in(0,1)$. We give an introduction to the theory of WIS integration before we examine existence and uniqueness of a solution to the SDE. We then introduce our numerical method which is based on previous theoretical results for $H\geq \frac{1}{2}$. We construct explicitly a translation operator required for the practical implementation of the method and are not aware of any other implementation of a numerical method for the WIS SDE. We then prove a strong convergence result that gives, in the autonomous case, an error of $O(\Dt^H)$ and in the non-autonomous case $O(\Dt^{\min(H,\zeta)})$, where $\zeta$ is a time-H\"older continuity parameter. We present some numerical experiments and conjecture that the theoretical results may not be optimal since we observe numerically a rate of $\min(H+\frac{1}{2},1)$ in the autonomous case. This work opens up the possibility to efficiently simulate SDEs for all $H$ values, including small values of $H$ when the stochastic integral is interpreted in the WIS sense.\\


\end{abstract}

%

\maketitle

\section{Introduction}\label{sec: intro}

Over the last three decades, fractional Brownian motion (fBm) and stochastic differential equations (SDEs) driven by fBm have received a lot of attention. 
FBm is a generalization of standard Brownian motion and is governed by a one-dimensional parameter $H\in(0,1]$ called the \emph{Hurst parameter}. If $H>\frac{1}{2}$, increments of fBm have positive correlation and if $H<\frac{1}{2}$, negative correlation, whereas $H=\frac{1}{2}$ corresponds to independent increments and fBm reduces to the standard Brownian motion. 
We consider in this paper the numerical approximation of a quasilinear SDE with multiplicative noise for all $H\in(0,1]$ of the form
\begin{equation}
\label{eq: quasi-linear SDE}
dX(t)=\alpha X(t)dt+a(t,X(t))dt+\beta X(t)dB^H(t), \quad X(0)=x_0, \quad t\in[0,T],
\end{equation}
where $\alpha,\beta,x_0\in\R$, $T>0$ and $a\colon[0,T]\times\R\to\R$ satisfies a global Lipschitz assumption in its second component (see Assumption~\ref{ass:a_existence}) and a H\"older continuity assumption in its first (see Assumption~\ref{ass:a_numerics}). We interpret the stochastic integral in the Wick-It\^o-Skorohod (WIS) sense which is well defined and centered for all $H\in(0,1]$.

SDEs forced with fBm have been applied to a large variety of topics such as  modelling of epidemic diseases \cite{Akinlar2020}, the pricing of weather derivatives \cite{Benth2003,Benth2013,Brody2002}, hydrology \cite{Lu2003}, weather forecasting \cite{Rivero2016}, modelling of electricity prices \cite{PrakasaRao2016,Simonsen2003}, neuroscience \cite{Coutin2002} and in the field of financial mathematics \cite{Bender2008,Czichowsky2017,Garcin2022,Guasoni2021,Han2023,Hu2003, Rogers1997}.  It is remarkable, however, that there are very few results on the numerical approximation of SDEs of the form \eqref{eq: quasi-linear SDE} for all $H \in (0,1]$. 
There is a relatively large body of work on the numerical methods for SDEs forced by additive fBm, when it is straightforward to interpret the stochastic integral. For multiplicative noise, to make sense of \eqref{eq: quasi-linear SDE}, we need to understand the stochastic integral
\begin{equation*}
    \int_0^T Y(t)d\fbm{t},
\end{equation*}
where $Y(t)$ is some stochastic process. Since $\fbm{t}$ is not a semimartingale for $H\neq\frac{1}{2}$, this is not straightforward. The first step was taken by Lin \cite{Lin1995}, who aimed to define the integral as the limit of Riemann sums. This is the basic principle behind the definition of the It\^o integral for standard Brownian motion. For $H>\frac{1}{2}$, the limit exists pathwise, but typically gives a random variable which is not centered. In the case of $H<\frac{1}{2}$, the Riemann sums often fail to converge, and thus a valid integral cannot be defined this way. The idea of taking the limit of Riemann sums has been captured in the notion of the \emph{pathwise integral}, introduced in a more general setting by Russo and Vallois \cite{Russo1993} and discussed in the fBm setting in \cite{Nualart2003}. As shown in \cite{Cheridito2005}, this approach extends the integral up to $H>\frac{1}{6}$, but for $H\in(\frac{1}{6},\frac{1}{2})$ the integral only exists in a weak sense. Elliot and van der Hoek \cite{Elliott2003} developed the  \emph{Wick--It\^o--Skorohod integral} or \emph{WIS integral}, which is a centered integral defined for all $H\in (0,1]$. It defines the integral on a particular space called the \emph{white noise probability space}. 
An extensive overview of the all different integrals and their relations is given in \cite{Biagini2008} and \cite{Mishura2008}.

There is some work available on numerical methods for SDEs driven by multiplicative fBm, many of which are for the pathwise integral in the case of $H>\frac{1}{2}$ \cite[see, for example,][]{Jamshidi2021,Mishura2008,Zhang2021}. We note the very recent work of \cite{10.1093/imanum/drad019}
who examine a stochastic PDE with multiplicative noise with a pathwise integral.
For $H\in(\frac{1}{2},\frac{1}{6})$, \cite{NeuenkirchNourdin} examines convergence of the Crank-Nicholson method for the pathwise Russo--Vallois integrals.
Recently, the idea to use rough path analysis to tackle the regime $H<\frac{1}{2}$ has received attention. This approach was started by \cite{Coutin2002} and is based on \cite{Lyons1994}. Examples for the regime $\frac{1}{3}<H<\frac{1}{2}$ can be found in \cite{Leon2023} and \cite{Liu2019}. We also note the related stochastic Volterra approach, which makes use of fractional kernels. Recent results include \cite{AlfonsiKebaier}, which examines kernel approximations for $H\in(0,\frac{1}{2})$.
However, the WIS integral remains largely unexplored when it comes to discussion and implementation of numerical methods. This work aims to contribute to opening up the use of the WIS integral in a numerical setting and hence in applications.

Notable exceptions are Chapter 3 of \cite{Mishura2008} and \cite{Mishura2008article} that consider the numerical approximation of SDEs of the form \eqref{eq: quasi-linear SDE}. In fact the analysis covers the more general SDE where $\beta$ may be non-autonomous (and $\alpha=0$). For $H\geq \frac{1}{2}$ a numerical method is proposed, that we denote \MISHURA, and a strong convergence result for autonomous $a$  is proved for $H\in [\frac{1}{2},1)$ of the form
\begin{equation}\label{eq:introstrong}
 \br{\E\left[\br{X(t_n)-X_n}^2\right]}^\frac{1}{2}\leq C_H\Dt^{H},  
\end{equation}
where $X_n\approx X(t_n)$ and $C_H>0$ is a constant that may depend on $H$. 
We introduce in this paper a numerical method, \GBMEM, based on that in \cite{Mishura2008}, that approximates \eqref{eq: quasi-linear SDE} and we prove a strong convergence result of the form \eqref{eq:introstrong} for all $H\in(0,1)$.  Further, taking $\beta$ constant and $\alpha=0$ we extend the convergence proof for \MISHURA to $H\in(0,\frac{1}{2})$.
 
Although a strong convergence result is proved in \cite{Mishura2008article,Mishura2008} the method does not seem to have ever been implemented, in part as it requires a non-trivial translation of the noise in the general case. We detail in Section~\ref{sec: quasilinear} how the methods \MISHURA and \GBMEM (as well as other methods)  may be implemented for the special case of \eqref{eq: quasi-linear SDE} and in Section~\ref{sec: Numerics}  illustrate their use. 
We observe from the numerical experiments that, for $H\geq \frac{1}{2}$, one may expect a theoretical convergence rate of order one, while for $H<\frac{1}{2}$ an even faster rate appears to occur numerically. For $H=\frac12$ the numerical method, \GBMEM, is equivalent to the Geometric Brownian Motion integrator proposed in \cite{ErdoganLord} where strong convergence of rate one was proved, and so the theoretical result is known to be sub-optimal for $H=\frac{1}{2}$.
 Based on numerous numerical experiments we conjecture a strong convergence result for autonomous $a$ of the form
\begin{equation}\label{eq:conjecture}
 \br{\E\left[\br{X(t_n)-X_n}^2\right]}^\frac{1}{2}\leq C_H\Dt^{\min(H+\frac{1}{2},1)}.  
\end{equation}

The outline of the paper is as follows. We start in Section~\ref{sec: fBm}  by providing a short recap on fBm before introducing white noise in Section~\ref{sec:whitenoise}, Wick calculus in Section~\ref{sec:WickCalculus} and the WIS integral in Section~\ref{sec:WISIntegration}. We then state the existence and uniqueness result for \eqref{eq: quasi-linear SDE} in Section~\ref{sec: quasilinear} before introducing the numerical method \GBMEM and proving convergence in Section~\ref{sec: convergence_proof}.  We then present our numerical experiments in Section~\ref{sec: Numerics}.  We conclude in Section~\ref{sec:discussion} with a discussion on the barrier we encountered in improving on this rate of convergence for $H\neq \frac{1}{2}$ and suggest some future research directions.

\section{Review of WIS integration}\label{sec:ReviewWIS}

We briefly recall some key properties of fBm, $B^H(t)$, before discussing how to interpret its derivative $W^H(t)=\frac{d}{dt} B^H(t)$ and some basic properties of Wick calculus. We are then in a position to introduce WIS integration. 
This material can be found, for example, in \cite{Mishura2008,Oksendal2007}.
\subsection{Fractional BM}
\label{sec: fBm}
Let $H\in(0,1]$. A one-dimensional stochastic process $(\fbm{t})_{t\geq0}$, is called a \emph{fractional Brownian motion} (fBm) with \emph{Hurst parameter} $H$ if it is a centered and continuous Gaussian process with the covariance function
\begin{equation*}
    \E[\fbm{t}\fbm{s}]=\frac{1}{2}\br{t^{2H}+s^{2H}-|t-s|^{2H}}.
\end{equation*}
This covariance function implies that $\fbm{0}=0$ almost surely and that $\E[\fbm{t}^2]=t^{2H}$. Moreover, its increments are stationary and $\fbm{s}-\fbm{t}\sim\mathcal{N}(0,|s-t|^{2H})$. For $H=\frac{1}{2}$ it is straightforward to check that fBm is in fact a standard Brownian motion. The notion of an fBm is therefore a generalization of that of standard Brownian motion, however for $H\neq\frac{1}{2}$, the increments of an fBm are not independent. 

Aside from the correlation of the increments, $H$ also governs the `roughness' of the sample paths. It can be shown that on any bounded interval $[0,T]$  fBm admits a modification that is almost surely H\"older continuous of order $\alpha$ for every $0<\alpha<H$. If we consider the process $X(t)=\xi t$, where $\xi\sim\mathcal{N}(0,1)$, then a straightforward computation shows that $X(t)$ is an fBm with $H=1$. Therefore, increased smoothness of the paths of an fBm culminates in $H=1$, where the paths are smooth. Since $H=1$ is trivial in this regard, we will from now on only consider $H\in(0,1)$.
We note that we can efficiently sample fBm, $B^H(t,\omega)$, numerically by using a circulant embedding method \cite[see, for example,][]{Lord2014}.

\subsection{White noise}
\label{sec:whitenoise}
FBm can be constructed on a specific probability space. Random variables on this probability with finite second moment can be decomposed in terms of a particular basis. This will provide us with a notion of a derivative of fBm, which is the main ingredient for the WIS integral.

Consider the Schwartz space of rapidly decreasing smooth functions, which is given by 
\begin{equation}
    \calS(\R)\coloneqq\{f\in C^\infty(\R)\mid \|f\|_{\alpha,\beta}<\infty, \,\forall\alpha,\beta\in\N\},
\end{equation}
where
\begin{equation*}
    \|f\|_{\alpha,\beta}\coloneqq\sup_{x\in\R}|x^\alpha f^{(\beta)}(x)|.
\end{equation*}
Let $\calS'(\R)$ be the dual space of $\calS(\R)$. Let $w\colon\calS'(\R) \times \calS(\R)\to\R$ be the mapping that outputs the action of $\omega\in\calS'(\R)$ on $f\in\calS(\R)$, so
$w(\omega,f)\coloneqq\omega(f)$.
In what follows $f\in\calS(\R)$ is often fixed, in which case the mapping is denoted as 
\begin{equation*}
    w_f(\omega)\coloneqq w(\omega,f).
\end{equation*}
Please be aware that in other literature the action of $\omega$ on $f$ is also often denoted as $\langle \omega,f\rangle$. 
We equip $\calS'(\R)$ with the weak-* topology and consider the Borel $\sigma-$algebra $\F=\mathcal{B}(\calS'(\R))$. 

From now on, we let $\Omega\coloneqq\calS'(\R)$. By the Bochner--Minlos theorem, see for example  \cite[Appendix A]{Holden1996}, there exists a probability measure $\P$ on $(\Omega,\F)$ such that 
\begin{equation}
    \label{eq: characteristic Schwartz}
    \E\left[\exp(iw_f)\right]=\int_\Omega \exp(iw_f(\omega))\P(d\omega)=\exp\left(-\frac{1}{2}||f||_{{L^2(\R)}}^2\right),\qquad \textup{for all } f\in\calS(\R).
\end{equation}
This measure $\P $ is called the \emph{white noise probability measure} and the triple $\left(\calS'(\R),\mathcal{B}(\calS'(\R)),\P \right)$ is called the \emph{white noise probablity space}.
Note that the mapping $\omega\mapsto w_f(\omega)$ is measurable, hence a random variable. Moreover, for any $f\in\calS(\R)$ we have that the characteristic function $\varphi_f\colon\R\to\mathbb{C}$ of ${w_{f}}$ is given by 
$$
\varphi_f(t)= \int_\Omega \exp(itw_f(\omega))\P(d\omega)=\int_\Omega \exp(iw_{tf}(\omega))\P(d\omega)\\
=\exp\left(-\frac{1}{2}t^2||f||_{{L^2(\R)}}^2\right).
$$
By the uniqueness of the characteristic function, this means that for any $f\in\calS(\R)$ the random variable ${w_{f}}$ is normally distributed with mean $0$ and variance $\|f\|^2_{L^2(\R)}$.

Since $\E[w_f^2]=\|f\|^2_{L^2(\R)}$, the mapping from $\calS(\R)$ to ${L^2(\P)}$ given by $f\mapsto {w_{f}}$ is an isometry. Combined with the fact that $\calS(\R)$ is dense in ${L^2(\R)}$ this mapping can be extended to a mapping from ${L^2(\R)}$ to ${L^2(\P)}$ with $${w_{f}}=\lim_{n\to\infty} w_{f_n},$$ where the limit is taken in ${L^2(\P)}$ and $(f_n)_{n\in\N}\subset\calS(\R)$ is such that $f_n\to f$ in ${L^2(\R)}$. Equation \eqref{eq: characteristic Schwartz} still holds for general $f\in{L^2(\R)}$. 
As a consequence, for any $f\in{L^2(\R)}$, ${w_{f}}$ is a normally distributed random variable with mean $0$ and variance $\|f\|_{L^2(\R)}^2$. The fact that ${w_{f}}$ is normally distributed for any $f\in{L^2(\R)}$ is exactly why $(\Omega,\F,\P)$ was constructed according to \eqref{eq: characteristic Schwartz}.

Random variables $X\in{L^2(\P)}$ on the white noise probability space can be decomposed in terms of a particular basis using Hermite polynomials and Hermite functions. 
For $n\in\N_0$, the \emph{Hermite polynomial} of order $n$ is defined as 
\begin{equation}
\label{eq: hermite poly}
    h_n(x)\coloneqq(-1)^ne^{x^2/2}\frac{d^n}{dx^n}\left(e^{-x^2/2}\right).
\end{equation}
The first few Hermite polynomials are easily computed to be
\begin{equation*}
    h_0(x)=1,\qquad h_1(x)=x,\qquad h_2(x)=x^2-1,\qquad h_3(x)=x^3-3x,\qquad\cdots .
\end{equation*}

The \emph{Hermite functions} of order $n$ are consequently defined for $n\geq 1$ as 
\begin{equation}
\label{eq: hermfunc}
    \xi_n(x)\coloneqq\pi^{-1/4}((n-1)!)^{-1/2}h_{n-1}(\sqrt{2}x)e^{-x^2/2}.
\end{equation}

It is well known that the set of Hermite functions is an orthonormal basis of ${L^2(\R)}$ \cite[see, for example,][Section 1.1]{Thangavelu1993}.  
Hermite functions eventually decay exponentially, so that $\xi_n\in\calS(\R)$ for every $n\geq 1$. 

Using the notation of \cite{Holden1996}, let $\J\coloneqq (\N_0^\N)_c$ be the set of multi-indices of arbitrary length, i.e. the set of all multi-indices $\alpha=(\alpha_1,\alpha_2,\dots)$ with only finitely many $\alpha_i\neq0$. The largest $n\in\N$ such that $\alpha_n\neq0$ is called the \emph{length} of $\alpha$. We will usually suppress the tail of zeros for convenience and write $\alpha=(\alpha_1,\dots,\alpha_n)$. Given $\alpha=(\alpha_1,\dots,\alpha_n)\in\J$ we define the absolute value $|\alpha|$ and the factorial $\alpha !$ as  
\begin{equation*}
    |\alpha|\coloneqq\sum_{i=1}^n\alpha_i, \qquad\qquad \alpha!\coloneqq\prod_{i=1}^n\alpha_i!.
\end{equation*}

For every $\alpha=(\alpha_1,\dots,\alpha_n)\in\J$ we consider the random variable $\calH_\alpha\colon\Omega\to\R$ defined by
\begin{equation}
\label{eq: Halpha}
    \calH_\alpha(\omega)\coloneqq\prod_{i=1}^nh_{\alpha_i}\br{w_{\xi_i}(\omega)}=h_{\alpha_1}(w_{\xi_1}(\omega))h_{\alpha_2}(w_{\xi_2}(\omega))...h_{\alpha_n}(w_{\xi_n}(\omega)).
\end{equation}
Note that since $h_0(x)=1$, we have $\calH_0=1$. These random variables form an orthogonal basis of ${L^2(\P)}$ termed the \emph{Wiener--It\^o chaos expansion}. For a proof we refer to Theorem~2.2.3 in \cite{Holden1996}. 
\begin{theorem}[Wiener--It\^o chaos expansion]
\label{thm: chaos expansion}
    The set $\{\calH_\alpha\}_{\alpha\in\J}$ is orthogonal in ${L^2(\P)}$ with $\E[\calH_\alpha^2]=\alpha!$. Moreover, if we let $F\in L^2(\P)$, then there is a unique sequence $\{c_\alpha\}_{\alpha\in\J}$ in $\R$ such that 
    \begin{equation*}
    F=\sum_{\alpha\in\J}c_\alpha\calH_\alpha.
    \end{equation*}
    Here, the convergence of the sum is in ${L^2(\P)}$.
\end{theorem}

Recall the familiar expansion for standard Brownian motion $B$ when $H=\frac{1}{2}$
\begin{equation}
\label{eq:Bseries}
    B(t)=\sum_{k=1}^\infty\left[\int_0^t\xi_k(s)ds\right] \calH_{\epsilon^{(k)}},   
\end{equation}
where ${\epsilon^{(k)}}$ denotes the multi-index that is $1$ in the $k$th place and 0 elsewhere. We now seek to construct a Wiener--It\^o expansion for fBm on the white noise probability space (see \eqref{eq: fbm expansion} below). 
To achieve this we introduce an operator $M_H$ on ${L^2(\R)}$. 

Let the Fourier transform of $f\in\calS(\R)$ be given by 
\begin{equation*}
    (\F f)(y)\coloneqq \int_\R e^{-ixy}f(x)dx.
\end{equation*}
For $H\in(0,1)$, let the operator $\M$ on $f\in\calS(\R)$ be defined by
\begin{equation}
\label{eq: M definition}
    (\F\M f)(y)=a_H|y|^{\frac{1}{2}-H}(\F f)(y),
\end{equation}
where $a_H\coloneqq\left[\Gamma(2H+1)\sin(\pi H)\right]^\frac{1}{2}.$
By applying the inverse Fourier transform on both sides of \eqref{eq: M definition}, we see that it indeed defines $\M$ uniquely. 

\begin{proposition}
    \label{prop: Schwartz in LH}
    Suppose $H\in(0,1)$. Then for any $f\in\calS(\R)$ we have $\M f\in{L^2(\R)}$.
\end{proposition}
\begin{proof}
    The Fourier transform maps the Schwartz space onto itself. Combined with the fact that 
\begin{equation*}
    \int_{-\epsilon}^\epsilon |y|^pdy<\infty,
\end{equation*}
for any $\epsilon>0$ and $|p|<1$, we see that $\F\M f\in{L^2(\R)}$ for any $f\in\calS(\R)$, where $\F\M f$ is given by \eqref{eq: M definition}. Since the (inverse) Fourier transform maps ${L^2(\R)}$ onto ${L^2(\R)}$, we conclude that $\M f \in{L^2(\R)}$ for any $f\in\calS(\R)$.
\end{proof}

\begin{remark}\label{rem:Mhalf}
    For $H=\frac{1}{2}$ we have that $a_{\frac{1}{2}}=1$ and thus 
        $(\F\M f)(y)=(\F f)(y)$.
    Therefore, $M_{\frac{1}{2}}$ is the identity operator on ${L^2(\R)}$.
\end{remark}

Since the Fourier transform and its inverse can be extended to ${L^2(\R)}$, the operator $\M$ can be extended to 
\begin{equation*}
    {L_H^2(\R)}\coloneqq\{f\colon\R\to\R \,;\, |y|^{\frac{1}{2}-H}(\F f)(y)\in {L^2(\R)}\}.
\end{equation*}
The (inverse) Fourier transform maps ${L^2(\R)}$ onto itself, so that we can also write
\begin{equation*}
    {L_H^2(\R)}=\{f\colon\R\to\R \,;\, \M f\in{L^2(\R)}\}.
\end{equation*}
By Proposition~\ref{prop: Schwartz in LH} it is indeed true that $\calS(\R)\subset{L_H^2(\R)}$. We equip ${L_H^2(\R)}$ with the inner product 
\begin{equation*}
    \langle f,g\rangle_H\coloneqq \langle \M f,\M g\rangle_{L^2(\R)}.
\end{equation*}

It can be shown that for any choice of $t\in\R$ we have that $\M\ind{t}\in{L^2(\R)}$, which implies that $\ind{t}\in{L_H^2(\R)}$. See Example 4.1.4 in \cite{Biagini2008} for more details.

The operator $\M$ is designed such that the following two results hold \cite[see, for example,][]{Mishura2008}.
\begin{proposition}
\label{prop: IP is iso}
For $t,s\geq 0$  
we have
    \begin{equation*}
        \langle\ind{t},\ind{s}\rangle_H=\langle \M\ind{t},\M\ind{s}\rangle_{L^2(\R)}=\frac{1}{2}(t^{2H}+s^{2H}-|t-s|^{2H}).
    \end{equation*}
\end{proposition}

\begin{corollary}
\label{lemma: preitsafbm}
    For $t\geq s\geq0$ we have
    \begin{equation*}
        \|\M\mathbf{1}_{[s,t]}\|_{L^2(\R)}^2=(t-s)^{2H}.
    \end{equation*}
\end{corollary}

The operator $\M$ and Proposition~\ref{prop: IP is iso} give us the tools to introduce fBm to the white noise probability space and Wiener-It\^o chaos expansion.  Consider the process
\begin{equation}\label{eq:BH}
    \tilde B^H(t)\coloneqq w_{\M\ind{t}}.
\end{equation}
We have $\M\ind{t}\in{L^2(\R)}$, so that $\tilde B^H(t)$ is well-defined. The random variable $w_f$ has a normal distribution with mean $0$ for any $f\in{L^2(\R)}$. This means that $\tilde B^H(t)$ is a normal random variable with mean $0$. Recall that $f\mapsto w_f$ is an isometry. Therefore, using Proposition~\ref{prop: IP is iso}, we have 
\begin{equation*}
    \E[\tilde B^H(t)\tilde B^H(s)]=\langle \M\ind{t},\M\ind{s}\rangle_{L^2(\R)}=\frac{1}{2}(t^{2H}+s^{2H}-|t-s|^{2H}).
\end{equation*}
By Kolmogorov's continuity lemma there exists a continuous modification of $\tilde B^H(t)$ that is an fBm. As usual, we denote this modification as $B^H(t)$.

The Wiener--It\^o chaos expansion of $\fbm{t}$ is given by \cite[see][]{Mishura2003,Oksendal2007}
    \begin{equation}
        \label{eq: fbm expansion}
        \fbm{t}=\sum_{k=1}^\infty\left[\int_0^t(\M\xi_k)(s)ds\right]\calH_{\epsilon^{(k)}}.
    \end{equation}
This is consistent with the expansion in \eqref{eq:Bseries} for standard Brownian motion, recalling from Remark~\ref{rem:Mhalf} that $M_{\frac12}$ is the identity operator.

If we could exchange the summation and integration in \eqref{eq: fbm expansion}, we would obtain
\begin{equation*}
    \fbm{t}=\int_0^t\left[\sum_{k=1}^\infty(\M\xi_k)(s)\calH_{\epsilon^{(k)}}\right]ds.   
\end{equation*}
This suggests that the fBm would be differentiable and has as its derivative the process
\begin{equation}
    \label{eq: fractional white noise}
    W^H(t)\coloneqq\sum_{k=1}^\infty(\M\xi_k)(t)\calH_{\epsilon^{(k)}},
\end{equation}
called the \emph{fractional white noise process}.

Of course, the trajectories of $\fbm{t}$ are almost surely not differentiable \cite[Proposition 1.7.1]{Biagini2008}. The problem is rooted in the fact that $W^H(t)$ is not an element of ${L^2(\P)}$.  To make sure that the fBm actually has a derivative, we need to consider a larger space than ${L^2(\P)}$. This space is called the \emph{Hida distribution space} and is denoted by $(\calS)^*$ which can be interpreted as the dual of the \emph{Hida test function space} $(\calS)$. For brevity, we omit a thorough discussion, but see Appendix~\ref{sec:Hida} and, for example,  \cite{Holden1996}. We note however that the fractional white noise process $W^H(\cdot)\colon\R\to(\calS)^*$ maps to $(\calS)^*$ and the fBm interpreted as a mapping $B^H(\cdot)\colon\R\to(\calS)^*$ is differentiable, with 
    \begin{equation*}
        \frac{d}{dt}B^H(t)=W^H(t)  \quad \textup{ in } (\calS)^*.
    \end{equation*}

\subsection{Wick calculus}
\label{sec:WickCalculus}
The regular product is generally ill-defined on the Hida distribution space $(\calS)^*$. As a result, it is usually equipped with a different product. 

\begin{definition}
    Let $F=\sum_{\alpha\in\J}a_\alpha\calH_\alpha$ and $G=\sum_{\beta\in\J}b_\beta\calH_\beta$ be elements of $(\calS)^*$. The \emph{Wick product}, denoted by $\diamond$, is the bivariate mapping given by 
    \begin{equation*}
        F\diamond G=\sum_{\alpha,\beta\in\J}a_\alpha b_\beta\calH_{\alpha+\beta}.
    \end{equation*}
\end{definition}
\begin{remark}
    It is immediate that we can also write 
    \begin{equation*}
        F\diamond G=\sum_{\gamma\in\J}\left(\sum_{\alpha+\beta=\gamma}a_\alpha b_\beta\right)\calH_{\gamma}.
    \end{equation*}
\end{remark}

It can be shown that $(\calS)^*$ is closed under the Wick product, while ${L^2(\P)}$ is not \cite[for an example, see][Section 3.1]{Gjessing1993}. It is relatively easy to verify that the Wick product is associative, commutative and distributive over addition.

In some situations the Wick product can be expressed in terms of a regular product. For $c\in\R$ and $F=\sum_{\alpha\in\J}a_\alpha\calH_{\alpha}\in(\calS)^*$ we have 
\begin{equation} \label{eq:scalar wick}
    c\diamond F=\sum_{\alpha\in\J}ca_\alpha\calH_{\alpha}=c\sum_{\alpha\in\J}a_\alpha\calH_{\alpha}=c\cdot F.
\end{equation}
We denote powers with respect to the Wick product as $F^{\dm n}$, i.e. we inductively define
\begin{equation*}
    F^{\dm n}\coloneqq
    \begin{cases}
        F\dm F^{\dm (n-1)} &\textup{if } n\geq 2,\\
        F &\textup{if } n=1.
    \end{cases}
\end{equation*}
We set $F^{\dm 0}=1$ for convenience.

Suppose now that $f,g\in{L^2(\R)}$ and consider the random variables $w_f$ and $w_g$. A straightforward computation \cite[see, for example][Section 2.4]{Holden1996}, then shows that
\begin{equation*}
    w_f\dm w_g=w_f\cdot w_g-\langle f,g\rangle_{L^2(\R)}.
\end{equation*}
If we let $t,s\geq0$,  Proposition~\ref{prop: IP is iso} gives
\begin{align*}
    \fbm{t}\dm\fbm{s}&=w_{\M\ind{t}}\dm w_{\M\ind{s}}=w_{\M\ind{t}}\cdot w_{\M\ind{s}}-\langle\M\ind{t},\M\ind{s}\rangle_{{L^2(\R)}}\\
    &=\fbm{t}\cdot\fbm{s}-\frac{1}{2}\br{t^{2H}+s^{2H}-|t-s|^{2H}}.\qedhere
\end{align*}

We now state some useful results from \emph{Wick calculus}.
\begin{definition}
Suppose $f\colon\R\to\R$ is an analytic function with Taylor series $f(x)=\sum_{n=0}^\infty a_n x^{n}$. Then the \emph{Wick version} of $f$ evaluated at $X\in(\calS)^*$ is defined as
\begin{equation*}
    f^\dm(X)\coloneqq\sum_{n=0}^\infty a_n X^{\dm n},
\end{equation*}
as long as this sum converges in $(\calS)^*$.
\end{definition}
There exists a chain rule for Wick versions. 
\begin{theorem}[Wick chain rule]
Let $Z\colon \R\to(\calS)^*$ be continuously differentiable in $(\calS)^*$ and let $f\colon\R\to\R$ be analytic. Then 
\begin{equation*}
    \frac{d}{dt}f^\dm(Z(t))=(f')^\dm(Z(t))\dm\frac{d}{dt}Z(t).
\end{equation*}
\end{theorem}
For a proof, see \cite[Proposition 5.14]{DiNunno2009}. 

\begin{proposition}[Wick product rule]
    Let $X\colon \R\to(\calS)^*$ and $Y\colon\R\to(\calS)^*$ both be differentiable in $(\calS)^*$. Then $X(t)\dm Y(t)$ is also differentiable in $(\calS)^*$, with
    \begin{equation*}
        \frac{d}{dt}\br{X(t)\dm Y(t)}=\br{\frac{d}{dt}X(t)}\dm Y(t)+X(t)\dm\br{\frac{d}{dt}Y(t)}.
    \end{equation*}
\end{proposition}
The proof of this result is completely analogous to the proof of the product rule for differentiable functions on $\R$.

A Wick version that we are particularly interested in is the Wick exponential
\begin{equation}
    \wexp{X}=\sum_{n=0}^\infty \frac{1}{n!} X^{\dm n},
\end{equation}
where $X\in(\calS)^*$ is such that the sum converges in $(\calS)^*$. 

The Wick exponential shares some properties with the regular exponential. For example, if $F, G\in(\calS)^*$ are such that $\wexp{F}$, $\wexp{G}$ and $\wexp{F+G}$ exist, then we have
\begin{equation*}
    \wexp{F+G}=\wexp{F}\dm\wexp{G}.
\end{equation*}
Generally, the Wiener--It\^o chaos expansion of a random variable might not be known. In this case, calculating its Wick powers is problematic. Even if the expansion is readily available, the infinite sum might make it very complicated to evaluate the Wick exponential. Fortunately, the Wick exponential of Gaussian random variables in ${L^2(\P)}$ can be expressed in terms of a regular exponential. It is shown in \cite{Holden1996} that for $f\in{L^2(\R)}$ we have
\begin{equation}
\label{eq: wexp to exp}
    \wexp{w_f}=\exp\br{w_f-\|f\|_{L^2(\R)}^2}.
\end{equation}

\subsection{WIS integration}
\label{sec:WISIntegration}
Having introduced some Wick calculus, we can finally define an integral of a process $Y(t)$ with respect to $\fbm{t}$ for all $H\in(0,1)$. Recall the fractional white noise $W^H(t)$ introduced in \eqref{eq: fractional white noise}. Since it is the derivative of $\fbm{t}$ in $(\calS)^*$, it makes sense to use this to define the stochastic integral. To make the following definition rigorous, we need a notion of integrability in $(\calS)^*$. We will not discuss this notion of $(\calS)^*$-integrability here, but details can be found in \cite[Section 2.5]{Holden1996}. 

\begin{definition}
\label{Defn:WIS}
    Consider a stochastic process $Y\colon\R\to(\calS)^*$. Assume the Wick product $Y(t)\dm W^H(t)$ is $(\calS)^*$-integrable. Then Y is called \emph{Wick--Itô--Skorohod (WIS) integrable} and has the \emph{WIS integral}
    \begin{equation}
        \int_\R Y(t)dB^H(t)\coloneqq\int_\R Y(t)\dm W^H(t)dt.
    \end{equation}
\end{definition}

We define the WIS integral on a bounded interval $[a,b]$ as
\begin{equation*}
    \int_a^b Y(t)dB^H(t)\coloneqq\int_\R Y(t)1_{[a,b]}(t)dB^H(t).
\end{equation*}

Suppose $H=\frac{1}{2}$, so $B^{\frac{1}{2}}(t)=B(t)$ is a regular Brownian motion. Let $a,b\in\R$. If $Y\colon\R\to(\calS)^*$ is an $\F_t$-adapted process such that
    \begin{equation*}
        \int_a^b\E[Y(t)^2]dt<\infty,
    \end{equation*}
then $Y$ is both WIS and It\^o integrable over $[a,b]$, and the two integrals coincide, as shown in \cite[Section 2.5]{Holden1996}. The WIS integral can therefore be seen as an extension of the It\^o integral.

If $Y\colon\R\to(\calS)^*$ is such that $\int_\R Y(t)dB^H(t)\in{L^2(\P)}$, then 
\cite[Eq. (4.21)]{Biagini2008} says that
\begin{equation}
\label{eq: expectation WIS integral}
    \E\left[\int_\R Y(t)dB^H(t)\right]=0.
\end{equation}
Therefore, the WIS integral is centered.

Finally we end this section with the following It\^o formula for fBm from \cite{Bender2003}. 
\begin{theorem}[Fractional It\^o formula]
\label{thm: Ito formula}
    Suppose $H\in(0,1)$. Let $f\colon \R^2\to\R$ be such that $f\in C^{1,2}(\R^2)$, i.e. $f$ is once continuously differentiable in its first component and twice continuously differentiable in its second component. Furthermore, assume that 
    \begin{equation*}
        f(t,\fbm{t}), \qquad \int_0^t\frac{\partial f}{\partial s}(s,\fbm{s})ds, \qquad \int_0^t\frac{\partial^2 f}{\partial x^2}(s,\fbm{s})s^{2H-1}ds
    \end{equation*}
    are in ${L^2(\P)}$ for all $t\in\R$. In this case we have
    \begin{multline*}
        f(t,\fbm{t})=f(0,0)+\int_0^t\frac{\partial f}{\partial s}(s,\fbm{s})ds\\+\int_0^t\frac{\partial f}{\partial x}(s,\fbm{s})d\fbm{s}+H\int_0^t\frac{\partial^2 f}{\partial x^2}(s,\fbm{s})s^{2H-1}ds.
    \end{multline*}
\end{theorem}

\subsection{Geometric fBm} \label{sec:gfBm}
Now that we have a notion of integration with respect to fBm, we can consider SDEs with such an integral. Let us first consider the basic example of the geometric fBm (gfBm)
\begin{equation}
\label{eq: fGBM SDE}
    dX(t)=\alpha X(t)dt+\beta X(t)dB^H(t),
\end{equation}
where $\alpha,\beta\in\R$, $t\geq 0$ and the stochastic integral is a WIS integral. This notation is a shorthand for the integral equation
\begin{align}
\label{eq: integral gfbm}
    X(t)&=X(0)+\alpha \int_0^tX(s)ds+\beta \int_0^t X(s)dB^H(s) \nonumber \\
    &=X(0)+\alpha \int_0^tX(s)ds+\beta \int_0^t X(s)\dm W^H(s)ds.
\end{align}
We are working in $(\calS)^*$ due to the presence of the WIS integral. This means that a solution $X(t)$ is an element of $(\calS)^*$ for every $t\geq0$. Instead of the integral equation \eqref{eq: integral gfbm}, we can consider the differential equation
\begin{equation}
\label{eq: geometric in diff}
    \frac{d}{dt}X(t)=\alpha X(t)dt+\beta X(t)\dm W^H(t)=X(t)\dm\left(\alpha+\beta W^H(t)\right),
\end{equation}
where the derivative is taken in $(\calS)^*$. From the form of this differential equation and the fact that we are working in $(\calS)^*$, which is equipped with the Wick product, we can make the educated guess that the solution is given by
\begin{equation}
\label{eq: gfbm Wick}
    X(t)=X(0)\dm\wexp{\alpha t+\beta B^H(t)}.
\end{equation}
Note that 
\begin{equation*}
\frac{d}{dt}(\alpha t+\beta B^H(t))=\alpha+\beta W^H(t),
\end{equation*}
so using the Wick product rule and the Wick chain rule we get
\begin{align*}
    \frac{d}{dt}X(t)&=\frac{d}{dt}\left(X(0)\dm\wexp{\alpha t+\beta B^H(t)}\right)=X(0)\dm\br{\frac{d}{dt}\wexp{\alpha t+\beta B^H(t)}}\\
    &=X(0)\dm\wexp{\alpha t+\beta B^H(t)}\dm\left(\alpha+\beta W^H(t)\right)=X(t)\dm\left(\alpha+\beta W^H(t)\right).
\end{align*}
Therefore, $X(t)$ in \eqref{eq: gfbm Wick} is indeed a solution to \eqref{eq: geometric in diff}. 

We can now use \eqref{eq: wexp to exp} to rewrite the Wick exponential in $X(t)$.  Since from \eqref{eq:BH} $\fbm{t}=w_{\M\ind{t}}$, we have $\beta\fbm{t}=w_{\beta\M\ind{t}}$. Therefore,
\begin{align} \label{eq:expBh}
    \wexp{\beta B^H(t)}&=\wexp{w_{\beta\M\ind{t}}}=\exp\br{w_{\beta\M\ind{t}}-\frac{1}{2}\|\beta\M\ind{t}\|_{L^2(\R)}^2} \nonumber \\
    &=\exp\br{\beta w_{\M\ind{t}}-\frac{1}{2}\beta^2\|\M\ind{t}\|_{L^2(\R)}^2}=\exp\br{\beta\fbm{t}-\frac{1}{2}\beta^2t^{2H}}, 
\end{align}
where the norm is evaluated using Corollary~\ref{lemma: preitsafbm}.
Recall that from \eqref{eq:scalar wick} the Wick product reduces to the regular product when one of the factors is in $\R$. Therefore, $\wexp{\alpha t}=\exp(\alpha t)$. We can now rewrite the Wick exponential as
\begin{align*}
    \wexp{\alpha t+\beta B^H(t)}&=\wexp{\alpha t}\dm\wexp{\beta B^H(t)}=\exp(\alpha t)\exp\left(\beta B^H(t)-\frac{1}{2}\beta^2t^{2H}\right)\\
    &=\exp\left(\alpha t-\frac{1}{2}\beta^2t^{2H}+\beta B^H(t)\right).
\end{align*}
In the case of a deterministic initial value $X(0)=x_0\in\R$, the solution \eqref{eq: gfbm Wick} becomes
\begin{equation}
\label{eq: gfbm}
    X(t)=x_0\exp\left(\alpha t-\frac{1}{2}\beta^2t^{2H}+\beta B^H(t)\right).
\end{equation}
This process is called the \emph{geometric fractional Brownian motion} (gfBm) and for $H=\frac12$ coincides with the regular geometric Brownian motion. The fact that $X(t)$ is the solution to \eqref{eq: fGBM SDE} can also be verified using the fractional It\^o formula of Theorem \ref{thm: Ito formula}.

\subsection{Translations and Gjessing's Lemma}
To examine existence and uniqueness for a quasilinear SDE and to formulate a numerical method we require a technical result called Gjessing's Lemma.

Recall that we have the white noise probability space $\Omega=\calS'(\R)$ and the Schwarz space $\calS(\R)$. We can interpret ${L^2(\R)}$ as a subset of $\calS'(\R)$ by considering the embedding $f\hookrightarrow A_f$, where $A_f$ is the linear operator that acts on $g\in\calS(\R)$ as 
\begin{equation}
\label{eq: Schwartz embedding}
    A_f g\coloneqq \int_\R f(x)g(x)dx.
\end{equation}

Suppose $\omega_0\in\Omega=\calS'(\R)$. The translation operator $\T_{\omega_0}\colon(\calS)\to(\calS)$ is defined by
\begin{equation*}
    (\T_{\omega_0}X)(\omega)\coloneqq X(\omega+\omega_0), \qquad \textup{ for all } \omega\in\Omega.
\end{equation*}
Here $(\calS)$ is the Hida test function space (see Appendix~\ref{sec:Hida}).
Theorem~4.15 in \cite{Hida1993} shows that $\T_{\omega_0}$ indeed maps $(\calS)$ onto itself, and does so continuously, for all $\omega_0\in\Omega$. 

Let us now fix some $p>1$. We can extend the definition of the translation operator $\T_{\omega_0}$, $\omega_0\in\Omega$, to $L^p(\P)$ by simply setting $\br{\T_{\omega_0}X}(\omega)\coloneqq X(\omega+\omega_0)$ for all $\omega\in\Omega$, where $X\in L^p(\P)$. The translation operator generally does not map $L^p(\P)$ onto itself. However, if $f\in{L^2(\R)}\subset\calS'(\R)$, then \cite[Corollary 2.10.5]{Holden1996} shows that for all $q<p$ we have $\T_fX\in L^q(\P)$ if $X\in L^p(\P)$.
Note that the translation operator has an obvious inverse, since 
\begin{equation}
\label{eq: translation operator inverse}
    \br{\T_{-\omega_0}\br{\T_{\omega_0}X}}(\omega)=X(\omega).
\end{equation}

This translation operator can be used in certain cases to remove a Wick product. This result is crucial for evaluating quasilinear SDEs. The first version of the result was stated as Theorem~2.10 by \cite{Gjessing1993manu}. Later it was generalized by \cite{Benth2000}. For a proof of the result as stated below, we refer to Theorem~2.10.7 in \cite{Holden1996}.

\begin{theorem}[Gjessing's Lemma]
\label{thm: Gjessing's Lemma}
Let $p>1$. Suppose that $f\in{L^2(\R)}$ and $X\in L^p(\P)$. Then $\wexp{w_f}\dm X\in L^q(\P)$ for all $q<p$. Moreover, 
\begin{equation*}
    \wexp{w_f}\dm X=\wexp{w_f} \cdot (\T_{-f}X), \qquad \textup{a.s.}.
\end{equation*}
\end{theorem}

\section{Quasilinear SDE and a numerical method}\label{sec: quasilinear}

In this section we consider quasilinear SDEs of the form \eqref{eq: quasi-linear SDE} with fBm as driving noise. This is a generalization of the geometric fBm discussed in Section~\ref{sec:gfBm}. The SDE does not in general have a closed form solution, so that numerical methods are needed. The formulation of the numerical method is based on \cite[Section 3.4.2]{Mishura2008} where $\alpha=0$ and $\beta(t)$ can be time dependent.

\subsection{Quasilinear SDEs}
Thanks to Section~\ref{sec:WISIntegration} we can now interpret \eqref{eq: quasi-linear SDE} as the differential equation
\begin{equation}
\label{eq: quasi-linear SDE differential}
\frac{d}{dt}X(t)=\alpha X(t)+a(t,X(t))+\beta X(t)\dm W^H(t), \quad X(0)=x_0, \quad t\in[0,T] \quad \textup{ in } (\calS)^*.
\end{equation}
To ensure existence and uniqueness we make the following assumption on $a$.
\begin{assumption}
\label{ass:a_existence}
The function $a\colon \R\to\R$ satisfies the following two conditions. There exists a constant $C>0$ such that 
\begin{enumerate}
    \item\label{enum: cond1} \emph{Linear growth:} $|a(t,x)|\leq C(1+|x|)$ for all $x\in\R$, $t\in[0,T]$.
    \item\label{enum: cond2} \emph{Lipschitz continuity:} $|a(t,x)-a(t,y)|\leq C|x-y|$ for all $x,y\in\R$, $t\in[0,T]$.
\end{enumerate}
\end{assumption}

The quasilinear SDE in \eqref{eq: quasi-linear SDE differential} has a unique solution, as shown in Section 3.4.2 of \cite{Mishura2008} using the ideas of Section 3.6 of \cite{Holden1996}.

\begin{theorem}
\label{thm: quasilinear SDE}
    The quasilinear SDE \textup{\eqref{eq: quasi-linear SDE differential}} 
    satisfying Assumption~\ref{ass:a_existence}
    has a unique solution $X(t)$ on $[0,T]$. Moreover, we have $X(t)\in L^p(\R)$ for all $p\geq 1$, $t\in[0,T]$.
\end{theorem}
The proof of this theorem involves some important concepts that we exploit while constructing a numerical method and hence we present a sketch of the proof here. For more details see, for example, \cite{Holden1996}.

\begin{proof}
    Assume $X(t)$ is a solution of the quasilinear SDE \eqref{eq: quasi-linear SDE differential}.  
    Consider the integrating factor
    $$ J(t)\coloneqq \wexp{-\alpha t-\beta B^H(t)}=\wexp{-\alpha t}\dm\wexp{-\beta \fbm{t}}.$$
By \eqref{eq:BH} we have $\fbm{t}=w_{\M\ind{t}}$ and $\beta \fbm{t}=w_{\beta \M\ind{t}}$, so we have both 
\begin{equation}
\label{eq:Jwickexp}
J(t) =\exp\br{-\alpha t}\wexp{w_{-\beta\M\ind{t}}} 
\end{equation}
and by \eqref{eq:expBh} that $J(t)$ and its inverse $J(t)^{-1}$ are given by 
\begin{equation}
\label{eq:JexpJinvexp}
J(t) =\exp\br{-\alpha t-\frac{1}{2}\beta^2 t^{2H}-\beta \fbm{t}}, \qquad J(t)^{-1}=\exp\br{\alpha t+\frac{1}{2}\beta t^{2H}+\beta B^H(t)}. 
\end{equation}    
    Define $Z(t)\coloneqq J(t)\dm X(t)$. By the Wick chain rule we have
    \begin{equation*}
        \frac{d}{dt}J(t)=J(t)\dm \br{-\alpha-\beta W^H(t)}.
    \end{equation*}
    Hence, by the Wick product rule and the fact that $X(t)$ solves \eqref{eq: quasi-linear SDE differential} we have
    \begin{align*}
        \frac{d}{dt}Z(t)&=\br{\frac{d}{dt}J(t)}\dm X(t)+J(t)\dm \br{\frac{d}{dt}X(t)}\\
        &=\br{-\alpha-\beta W^H(t)}\dm J(t) \dm X(t)+J(t)\dm\br{\alpha X(t)+ a(t,X(t))+\beta X(t)\dm W^H(t)}\\
        &= J(t)\dm a(t,X(t)).
    \end{align*}
Since $\ind{t}\in{L_H^2(\R)}$, we have that $-\beta\M\ind{t}\in{L^2(\R)}$. This means that we can use $J(t)$  as in \eqref{eq:Jwickexp} and Gjessing's Lemma to obtain
\begin{align*}
    J(t)\dm a(t,X(t))&=\exp\br{-\alpha t}\wexp{w_{-\beta\M\ind{t}}}\dm a(t,X(t))\\
    &=J(t)\br{\T_{\beta\M\ind{t}}a(t,X(t))}
\end{align*}
and
\begin{align*}
    Z(t)&=\exp\br{-\alpha t}\wexp{w_{-\beta\M\ind{t}}}\dm X(t)\\
    &=J(t)\br{\T_{\beta\M\ind{t}}X(t)}.
\end{align*} 
Combining \eqref{eq:JexpJinvexp} with the fact that the translation operator can also be inverted as in \eqref{eq: translation operator inverse}, we can write
\begin{equation}
\label{eq: X in Z}
    X(t)=\T_{-\beta\M\ind{t}}\left[J(t)^{-1}Z(t)\right].
\end{equation}
Putting this all together we find 
\begin{align*}
    \frac{d}{dt}Z(t)&=J(t)\dm a(t,X(t))\\
    &=J(t)\br{\T_{\beta\M\ind{t}}a(t,X(t))}\\
    &=J(t)\br{\T_{\beta\M\ind{t}}a(t,\T_{-\beta\M\ind{t}}\left[J(t)^{-1}Z(t)\right])}\\
    &=J(t)a(t,J(t)^{-1}Z(t)).
\end{align*}
Therefore, we have the following differential equation for $Z(t)$:
\begin{equation}
\label{eq: Z differential}
    \frac{d}{dt}Z(t)=J(t)a(t,J(t)^{-1}Z(t)), \qquad Z(0)=x_0.
\end{equation}
This equation has a unique solution $Z(t,\omega)$ on $[0,T]$ for any $\omega\in\Omega$ due to the assumptions on $a(t,x)$ \cite[see][(3.6.15)]{Holden1996}. Together with \eqref{eq: X in Z}, this implies that our original SDE \eqref{eq: quasi-linear SDE differential} has a unique solution $X(t,\omega)$ for any $\omega\in\Omega$. The second statement of the theorem follows from the proof of Lemma~\ref{lem:J_BoundedMoments} in the next section and can be found in Section 3.3 of \cite{Mishura2008} as well.
\end{proof}

\subsection{A numerical method for quasilinear SDEs}
\label{subsec: numerical method}
The proof of Theorem~\ref{thm: quasilinear SDE} contains some elements that help us formulate a numerical method for quasilinear SDEs driven by fBm. Our method is an adaptation of that from Section 3.4.2 of \cite{Mishura2008}.

Suppose we want to sample an approximation to the solution $X(t)$ of a quasilinear SDE. We can sample a path of the underlying fBm $\fbm{t,\omega}$ using, for example, a circulant embedding \cite[see, for example,][]{Lord2014}. Afterwards, we can numerically approximate the process $Z(t,\omega)$ given by the differential equation \eqref{eq: Z differential}. This equation does not contain any Wick products 
and can therefore be approximated using standard methods.  To obtain a pathwise approximation of $X(t,\omega)$ given by \eqref{eq: X in Z}, however, we need to deal with the translation operator $\T_{-\beta\M\ind{t}}$. Fortunately, this particular translation operates nicely on $\fbm{t}$. 

Let $t,s\in[0,T]$. By the fact that $\fbm{s,\omega}=w_{\M\ind{s}}(\omega)=\omega(\M\ind{s})$ for all $\omega\in\calS'(\R)$, together with the embedding ${L^2(\R)}\hookrightarrow\calS'(\R)$ given by \eqref{eq: Schwartz embedding}, we find that
\begin{align*}
\label{eq: translated fBm}
    \T_{-\beta\M\ind{t}}\fbm{s,\omega}&=\fbm{s,\omega-\beta\M\ind{t}}\\
    &=(\omega-\beta\M\ind{t})(\M\ind{s})\\
    &=\omega(\M\ind{s})-\beta \int_\R(\M\ind{t})(x)(\M\ind{s})(x)dx\\
    &=w_{\M\ind{s}}(\omega)-\beta\langle\M\ind{t},\M\ind{s}\rangle_{L^2(\R)}.
\end{align*}
Proposition~\ref{prop: IP is iso} then gives us
\begin{equation*}
    \T_{-\beta\M\ind{t}}\fbm{s}=\fbm{s}-\frac{1}{2}\beta(t^{2H}+s^{2H}-|t-s|^{2H}).
\end{equation*}
This result can be used to evaluate the factors of the form $\T_{-\beta\M\ind{t}}J(s)$, such as those that appear in \eqref{eq: X in Z} of the proof of Theorem~\ref{thm: quasilinear SDE}:
\begin{align*}
    \T_{-\beta\M\ind{t}}J(s)=&\exp\br{-\alpha s-\frac{1}{2}\beta^2 s^{2H}-\beta \T_{-\beta\M\ind{t}}\fbm{s}}\\
    =&\exp\br{-\alpha s-\frac{1}{2}\beta^2 s^{2H}-\beta \fbm{s}+\frac{1}{2}\beta^2(t^{2H}+s^{2H}-|t-s|^{2H})}\\
    =&\exp\br{-\alpha s+\frac{1}{2}\beta^2(t^{2H}-|t-s|^{2H})-\beta\fbm{s}}.
\end{align*}
For brevity we consider for some fixed $t\in[0,T]$ the processes $\tJt\colon [0,t]\to(\calS)^*$ and $\tZt\colon[0,t]\to(\calS)^*$ defined by
\begin{equation}
\label{eq: J in omega t}
    \tJt(s)\coloneqq \T_{-\beta\M\ind{t}}J(s)=\exp\br{-\alpha s+\frac{1}{2}\beta^2(t^{2H}-|t-s|^{2H})-\beta\fbm{s}},
\end{equation}
as well as
\begin{equation}
\label{eq: tilde Z def}
    \tZt(s)\coloneqq \T_{-\beta\M\ind{t}}Z(s).
\end{equation}
In particular, Equation \eqref{eq: X in Z} then becomes 
\begin{equation}
    \label{eq: X(t) expression}
    X(t)=\tJt(t)^{-1}\tZt(t),
\end{equation}
where 
\begin{equation*}
    \tJt(t)^{-1}=\exp\br{\alpha t-\frac{1}{2}\beta^2t^{2H}+\beta\fbm{t}}.
\end{equation*}
To evaluate $X(t,\omega)$, it now suffices to know the value of $\tZt(t,\omega)=Z(t,\omega-\beta\M\ind{t})$. Since the latter satisfies the differential equation \eqref{eq: Z differential}, we have that $\tZt(s)$ is given by 
\begin{equation}
\label{eq: tilde Z differential}
    \frac{d}{ds}\tZt(s)=\tJt(s)a(s,\tJt(s)^{-1}\tZt(s)), \qquad \tZt(0)=x_0
\end{equation}
on $[0,t]$.
Therefore, to approximate the value $X(t,\omega)$, we first need to approximate the path of $\tZt(s,\omega)$ on $[0,t]$ governed by \eqref{eq: tilde Z differential} using the underlying fBm path $\fbm{s,\omega}$. 

Let us for now focus on approximating $X(T,\omega)$ at the end time $T$. We take a uniformly spaced grid $0=t_0<t_1<\dots<t_n=T$, where $t_{i}=i\Delta t$ for some fixed $\Dt>0$. Since $\tZT(s,\omega)$ satisfies \eqref{eq: tilde Z differential}, we have $\tZT(t_0)=x_0$ and 
\begin{equation}
\label{eq: integral Z}
    \tZT(t_{i+1},\omega)=\tZT(t_i,\omega)+\int_{t_i}^{t_{i+1}}\tJT(s,\omega)a(s,\tJT(s,\omega)^{-1}\tZT(s,\omega))ds.
\end{equation}
There are several ways to approximate the integral. Application of the explicit Euler method gives
\begin{equation}
\label{eq:explicit Euler approximation}
    \tZT(t_{i+1},\omega)\approx \tZT(t_i,\omega)+\Delta t\tJT(t_i,\omega)a(t_i,\tJT(t_i,\omega)^{-1}\tZT(t_i,\omega)).
\end{equation}
This results in the iterative numerical scheme 
\begin{equation*}
    \tZ^n_0=x_0,
\end{equation*}
\begin{equation}
\label{eq: Z numerical scheme}
    \tZ^n_{i+1}=\tZ_i^n+\Delta t \tJ_i^{n}a(t_i,\tZ_i^n/\tJ_i^n), \qquad 0\leq i\leq n-1,
\end{equation}
where 
\begin{align*}
    \tJ_i^n\coloneqq \tJT(t_i,\omega)&=\exp\br{-\alpha t_i+\frac{1}{2}\beta^2\br{T^{2H}-(T-t_i)^{2H}}-\beta\fbm{t_i,\omega}}
\end{align*}

This numerical scheme gives us an approximation $\tZT(T,\omega)\approx \tZ_n^n$. 
Finally we obtain the approximation of the solution of the quasilinear SDE at time $T=t_n$:
\begin{equation}
\label{eq: Xn num method}
    X(T,\omega)\approx X_n \coloneqq  \tZ_n^n/\tJ_n^n
\end{equation}
This entire numerical scheme that we denote \GBMEM is summarized in Algorithm~\ref{alg: single time}.

\begin{algorithm}
\caption{Generate approximation of the solution $X(T)$ to the SDE \eqref{eq: quasi-linear SDE} at time $T$.\\
\textbf{Input:} Hurst parameter $H$, initial value $x_0$, final time $T$, number of steps $n$, fBm path $B^H\in \R^n$.\\
\textbf{Output:} A value $X_n\in\R$ such that $X_n\approx X(T)$.}
\label{alg: single time}
\begin{algorithmic}[1]
\State $\Delta t=T/n$
\State $\tZ^n_0=x_0$
\For{$i=0,\dots,n-1$}
    \State $\tJ^n_i=\exp(-\alpha i\Dt+\frac{1}{2}\beta^2\Delta t^{2H}(n^{2H}-(n-i)^{2H})-\beta B_i^H)$
    \State $\tZ^n_{i+1}=\tZ_i^n+\Delta t \tJ^n_ia(i\Delta t,\tZ^n_i/\tJ^n_i)$
\EndFor
\State $\tJ_n^n=\exp(-\alpha n\Dt+\frac{1}{2}\beta^2 (n\Delta t)^{2H}-\beta B_n^H)$
\State $X_n=\tZ_n^n/\tJ_n^n$
\end{algorithmic}
\end{algorithm}

Suppose now that we want to approximate the path of $X(t,\omega)$ on $[0,T]$. Unfortunately, for every single $t\in[0,T]$ for which we want to obtain a sample of $X(t,\omega)$, we have to approximate the value of the individual process $\tZt(s,\omega)$ at time $t$. We thus need to run the numerical scheme once each time for each approximation of $X(t,\omega)$.

We again take a uniformly spaced grid $0=t_0<t_1<\dots<t_n=T$, where $t_{i}=i\Delta t$ for some $\Dt>0$. We approximate the value $X(t_k,\omega)$ by running Algorithm $\ref{alg: single time}$ from $0$ to $t_k$ on the grid $0=t_0<t_1<\dots<t_k$. This results in Algorithm~\ref{alg: path sample} for \GBMEM.

\begin{algorithm}
\caption{Generate approximation of the path of solution X(t) to the SDE \eqref{eq: quasi-linear SDE} on $[0,T]$.\\
\textbf{Input:} Hurst parameter $H$, initial value $x_0$, final time $T$, number of steps $n$, fBm path $B^H\in\R^n$.\\
\textbf{Output:} A vector $X\in \R^n$ such that $X_k\approx X(t_k)$, $k=1,\ldots, n$.}
\label{alg: path sample}
\begin{algorithmic}[1]
\State $\Delta t=T/n$
\For{$k=1,\dots,n$}
    \State $\tZ^k_0=x_0$

    \For{$i=0,\dots,k-1$}
        \State $\tJ^k_i=\exp(-\alpha i\Dt+\frac{1}{2}\beta^2\Delta t^{2H}(k^{2H}-(k-i)^{2H})-\beta B^H_i)$
        \State $\tZ^k_{i+1}=\tZ^k_i+\Delta t \tJ^k_ia(i\Delta t,\tZ^k_i/\tJ^k_i)$
    \EndFor

    \State $\tJ^k_k=\exp(-\alpha k\Dt+\frac{1}{2}\beta^2 (k\Delta t)^{2H}-\beta B^H_k)$
    \State $X_k=\tZ^k_k/\tJ^k_k$
    
\EndFor
\end{algorithmic}
\end{algorithm}
\begin{remark}
    Simulating a full sample path using Algorithm~\ref{alg: path sample} is computationally expensive due to the fact that for each distinct time instance of the path we need to approximate a solution of an ODE up to this time using the explicit Euler scheme as in Algorithm~\ref{alg: single time}. This is caused by the fact that the solution is not Markovian and therefore the entire past needs to be taken into account to calculate a subsequent value of the approximation.
\end{remark}
\begin{remark}
    \label{rem:Mishura1}
    In the work of \cite{Mishura2008} a similar numerical method is proposed, that we will call \MISHURA. The difference with \GBMEM is that the linear drift term involving $\alpha$ is not accounted for in the integrating factor $J(t)$ as in \eqref{eq:JexpJinvexp}, but is discretized together with the non-linear drift term $a(t,x)$ using an explicit Euler approximation as in \eqref{eq:explicit Euler approximation}. In the case of gfBm \eqref{eq: fGBM SDE}, where $a(t,x)=0$, \MISHURA therefore gives an approximation of the solution using an explicit Euler scheme, whereas \GBMEM solves it exactly. If $\alpha=0$, the two methods coincide. The convergence proof of \MISHURA in \cite{Mishura2008} relies on results only valid for $H>\frac{1}{2}$ in particular to deal with the more general setting of a non-autonomous forcing term $\beta(t)$.
\end{remark}

To prove our convergence result we make the following additional assumption on the regularity of $a(t,X)$.
\begin{assumption}
    \label{ass:a_numerics}
We assume there exists constants $C>0$ and $\zeta>0$ such that for all $t,s\in[0,T]$ and $x\in\R$
$$|a(t,x)-a(s,x)|\leq C (1+\left| x\right|)|t-s|^\zeta.$$ 
\end{assumption}
We now state our main result.
\begin{theorem}
\label{thm: RMSE}
    Let the function $a\colon[0,T]\times\R\to\R$ satisfy Assumptions~\ref{ass:a_existence}
    and \ref{ass:a_numerics} and let $X(t)$ be the solution to \textup{\eqref{eq: quasi-linear SDE differential}}. Let $X_n$ be given by \textup{\eqref{eq: Xn num method}}. Then there exists a constant $C_H>0$ such that for all $H\in(0,1)$
    \begin{equation*}
        \br{\E\left[\br{X(t_n)-X_n}^2\right]}^\frac{1}{2}\leq C_H \Dt^{\min(H,\zeta)}.
    \end{equation*}
\end{theorem}

\begin{remark}
    We note that if $a(x)\geq 0$, $x_0>0$ then the numerical method constructed in \eqref{eq: Xn num method} would preserve positivity of the solution, i.e. $X_n>0$.
\end{remark}

\begin{remark}
    The results of this section can be extended to SDEs with fractional additive noise of the from
    \begin{equation*}
        \frac{d}{dt}X(t)=\alpha X(t)+a(t,X(t))+(\beta X(t)+\sigma)\dm W^H(t), \quad X(0)=x_0, \quad t\in[0,T] \quad \textup{ in } (\calS)^*,
    \end{equation*}
    where $\sigma\in\R$ and $\beta\neq0$, by considering the process $Y(t)=X(t)+\frac{\sigma}{\beta}$ similarly to the standard Brownian motion case $H=\frac{1}{2}$. The process $Y$ then solves the differential equation
    \begin{equation*}
        \frac{d}{dt}Y(t)=\alpha Y(t)-\frac{\alpha\sigma}{\beta}+a(t,Y(t)-\frac{\sigma}{\beta})+\beta Y(t)\dm W^H(t), \quad Y(0)=x_0+\frac{\sigma}{\beta}, \quad t\in[0,T] \quad \textup{ in } (\calS)^*,
    \end{equation*}
    which is again of the of the form \eqref{eq: quasi-linear SDE differential}.
\end{remark}

\section{Convergence Proof}
\label{sec: convergence_proof}

 We start by reviewing some preliminary results. For convenience we let $C>0$ and $C_H>0$ denote constants that may change from line to line, where the latter may depend on $H$, but both are independent of $\Dt$. From now on $a$ satisfies both Assumption~\ref{ass:a_existence} and \ref{ass:a_numerics}.

\subsection{Preliminary Lemmas}
First we give four preliminary lemmas.
\begin{lemma}
\label{lemma: 2H power difference}
    Let $t,s\in[0,T]$. Then
    \begin{equation}
        |t^{2H}-s^{2H}|\leq 
        \begin{cases}
         |t-s|^{2H},& \text{if } H\leq \frac{1}{2},\\
         2HT^{2H-1}|t-s|,& \text{if } H> \frac{1}{2}.
        \end{cases}
    \end{equation}
\end{lemma}
\begin{proof}
    Assume without loss of generality that $t\geq s>0$. The statement for $H>\frac{1}{2}$ is a direct consequence of the mean value theorem: there exists a constant $c\in(s,t)$ such that
    \begin{equation*}
        t^{2H}-s^{2H}=2Hc^{2H-1}(t-s).
    \end{equation*}
    Using the fact that $c\leq t\leq T$ then gives the desired result. 
    
    For $H\leq\frac{1}{2}$ we consider the function $f\colon [1,\infty)\to \R$ given by $f(x)\coloneqq x^{2H}-1-(x-1)^{2H}$. This function is differentiable on $(1,\infty)$ and has the derivative
    \begin{equation*}
        f^\prime(x)=2H(x^{2H-1}-(x-1)^{2H-1}).
    \end{equation*}
    Since $x\mapsto x^{2H-1}$ is decreasing on $(0,\infty)$ for $H\leq\frac{1}{2}$, we find that $f^\prime(x)\leq 0$ for all $x>1$. Combined fact that $f(1)=0$, we can conclude that $f(x)\leq 0$ for all $x\geq 1$. Inserting $x=\frac{t}{s}$ into $f(x)$ then gives
    \begin{equation*}
        \br{\frac{t}{s}}^{2H}-1\leq \br{\frac{t}{s}-1}^{2H},
    \end{equation*}
    from which the result follows by multiplication of both sides with $s^{2H}$.
\end{proof}

We now consider the processes $\tZT(t)$ and $\tJT(t)$ given by 
\begin{equation}\label{def:ZJtilde}
  \tZT(t)= \T_{-\beta\M\ind{T}}Z(t),  
\qquad 
\tJT(t)= \T_{-\beta\M\ind{T}}J(t),
\end{equation}
as defined in \eqref{eq: J in omega t} and \eqref{eq: tilde Z def}.

In particular, by \eqref{eq: J in omega t} we have
\begin{equation}
\label{eq: J tilde}
\tJT(t)=\exp\br{-\alpha t+\frac{1}{2}\beta^2(T^{2H}-(T-t)^{2H})-\beta\fbm{t}}.
\end{equation}

It can be shown that $\tJT(t)$ and $\tZT(t)$ have uniformly bounded moments on $[0,T]$. For $\tJT(t)$ this follows directly from the fact that it is an exponential of a normal random variable with bounded variance on $[0,T]$. 

 \begin{lemma}
 \label{lem:J_BoundedMoments}
 For every $q\geq 1$ there exists a constant $C_H>0$ such that for all $t\in[0,T]$ we have 
 $$\E[\tJT(t)^q]\leq C_H, \quad \E[(\tJT(t)^{-1})^q]\leq C_H, \quad \E[|\tZT(t)|^q]\leq C_H.$$
 \end{lemma}
\begin{proof}

Fix some $q\geq 1$. In \eqref{eq: J tilde} we see that $\tJT(t)$ is the exponent of a Gaussian random variable. Since the moment generating function of a Gaussian random variable $X\sim\mathcal{N}(\mu,\sigma^2)$ is given by $\E[e^{qX}]=e^{\mu q+\frac{1}{2}\sigma^2q^2}$, it follows that 
\begin{align*}
    \E\left[\tJT(t)^q\right]&=\exp\br{-\alpha qt+\frac{1}{2}q\beta^2(T^{2H}-|T-t|^{2H})+\frac{1}{2}q^2\beta^2t^{2H}}\\
  &  \leq \exp\br{|\alpha|qT+q^2\beta^2T^{2H}}.
\end{align*}
The statement for $\tJT(t)^{-1}$ follows similarly. For the uniform bound on $\E[|\tZT(t)|^q]$ we closely follow the second part of the proof of Theorem~3.6.1 in \cite{Holden1996}. By \eqref{eq: tilde Z differential} and the linear growth of $a(t,x)$ in $x$ we have for all $t\in[0,T]$ that
\begin{align*}
    |\tZT(t)|&\leq |x_0|+\int_0^t\left|\tJT(s)a(s,\tJT(s)^{-1}\tZT(s))\right|ds \\ 
    & \leq |x_0|+C\int_0^t\tJT(s)\br{1+\tJT(s)^{-1}\left|\tZT(s)\right|}ds\\
    & = |x_0|+C\int_0^T\tJT(s)ds+C\int_0^t\left|\tZT(s)\right|ds.
\end{align*}
Positivity of $\tJT(s)$ and Gronwall's inequality then gives the bound 
\begin{equation*}
    |\tZT(t)|\leq\br{|x_0|+C\int_0^T\tJT(s)ds}\exp\br{Ct},
\end{equation*}
so that
\begin{equation*}
    |\tZT(t)|^q\leq2^{q-1}\exp\br{Cqt}\br{|x_0|^q+C\br{\int_0^T\tJT(s)ds}^q}.
\end{equation*}
Applying Jensen's inequality for integrals we get
\begin{align*}
    \E[|\tZT(t)|^q]&\leq\E\left[2^{q-1}\exp\br{Cqt}\br{|x_0|^q+CT^{q-1}\int_0^T\tJT(s)^qds}\right]\\
    &\leq2^{q-1}\exp\br{CqT}\br{|x_0|^q+CT^{q-1}\int_0^T\E[\tJT(s)^q]ds}\\
    &\leq2^{q-1}\exp\br{CqT}\br{|x_0|^q+CT^{q}\exp\br{|\alpha|qT+q^2\beta^2T^{2H}}}.\qedhere
\end{align*}
\end{proof}

\begin{lemma}    
\label{lem:Zcontrol}
    Let $q\geq1$. Then $\E[|\tZT(t)-\tZT(s)|^q]\leq C_H|t-s|^q$ for any $t,s\in[0,T]$.
\end{lemma}
\begin{proof}
    Assume without loss of generality that $t\geq s$. By \eqref{eq: tilde Z differential}, Jensen's inequality for integrals and the stochastic Cauchy--Schwarz inequality, we have
    \begin{align*}
        \E[|\tZT(t)-\tZT(s)|^q]&=\E\left[\left|\int_{s}^t\tJT(r)a(r,\tJT(r)^{-1}\tZT(r))dr\right|^q\right]\\
        &\leq (t-s)^{q-1}\int_{s}^t\E\left[\left|\tJT(r)a(r,\tJT(r)^{-1}\tZT(r))dr\right|^q\right]\\
        &\leq (t-s)^{q-1}\int_{s}^t\E\left[\tJT(r)^{2q}\right]^\frac{1}{2}\E\left[\left|a(r,\tJT(r)^{-1}\tZT(r))\right|^{2q}\right]^\frac{1}{2}dr.
    \end{align*}
    By Lemma~\ref{lem:J_BoundedMoments} we have that $\E\left[\tJT(r)^{2q}\right]$ is uniformly bounded. The linear growth assumption on $a$ implies that 
    $$|a(r,\tJT(r)^{-1}\tZT(r))|\leq C(1+|\tJT(r)^{-1}\tZT(r)|).$$ 
    From the stochastic Cauchy--Schwarz inequality and the fact that $\tJT(r)^{-1}$ and $\tZT(r)$ both have uniformly bounded $4q$th moments, it follows that $\E\left[\left|a(r,\tJT(r)^{-1}\tZT(r))\right|^{2q}\right]$ is uniformly bounded as well. Hence, $\E[|\tZT(t)-\tZT(s)|^q]\leq C_H(t-s)^{q-1}(t-s)=C_H(t-s)^q$.
\end{proof}

\begin{lemma} \label{lemma: DJorder}
For $q\geq 1$ with $\tJT(t)$ defined in \eqref{eq: J tilde} we have the following two estimates for $r\in[t_i,t_{i+1}]$:
$$
    (i) \quad \expect{\left| \tJT(r) - \tJT(t_i) \right|^{2q}}\leq C_H\Dt^{2qH}, \qquad 
    (ii) \quad \expect{\left| \tJT(r)^{-1} - \tJT(t_i)^{-1} \right|^{2q}}\leq C_H\Dt^{2qH}.
$$    
\end{lemma}
\begin{proof} We only prove $(i)$ here as the  proof of $(ii)$ is similar. The proof presented here follows the same approach as Lemma 3.4.6 of \cite{Mishura2008}.

Note that for $u,v\in\R$ we have
\begin{equation}
\label{eq: exp diff bound}
    |e^u-e^v|\leq(e^u+e^v)|u-v|.
\end{equation}
This can be seen by assuming without loss of generality that $v<u$ and applying the mean value theorem to find a $z\in(v,u)$ such that $e^u-e^v=e^z(u-v)\leq (e^u+e^v)(u-v)$.

Let us now set $U=-\alpha r+\frac{1}{2}\beta^2(t_n^{2H}-|t_n-r|^{2H})-\beta\fbm{r}$ and   $V=-\alpha t_i+\frac{1}{2}\beta^2(t_n^{2H}-|t_n-t_i|^{2H})-\beta\fbm{t_i}$, so that $\tJT(r)=e^U$ and $\tJT(t_i)=e^V$. Consider the Gaussian variable
 $$
 U-V=-\alpha (r-t_i)+\frac{1}{2}\beta^2(|t_n-t_i|^{2H}-|t_n-r|^{2H})-\beta (\fbm{r}-\fbm{t_i}).
 $$
 This has second moment
 \begin{align*}
 \expect{(U-V)^2}&=\beta^2 (r-t_i)^{2H}+ \left(-\alpha (r-t_i)+\frac{1}{2}\beta^2(|t_n-t_i|^{2H}-|t_n-r|^{2H})\right)^2 \\
 &   \leq C_H \Dt^{2H},
 \end{align*}
where the final inequality is in part a consequence of Lemma~\ref{lemma: 2H power difference}. Using \eqref{eq: exp diff bound} and the stochastic Cauchy--Schwarz inequality we then find that
\begin{align*}
    \expect{\left| \tJT(r) - \tJT(t_i) \right|^{2q}}&=\expect{\left|e^U-e^V\right|^{2q}}\leq\expect{\left(e^U+e^V\right)^{2q}|U-V|^{2q}}\\
    &\leq\br{\expect{\left(\tJT(r)+\tJT(t_i)\right)^{4q}}}^\frac{1}{2}\br{\expect{|U-V|^{4q}}}^\frac{1}{2}.
\end{align*}
Using Lemma~\ref{lem:J_BoundedMoments} we see that $\expect{\left(\tJT(r)+\tJT(t_i)\right)^{4q}}\leq C_H$. Moreover, since $U-V$ is Gaussian, we have that
\begin{equation*}
    \expect{|U-V|^{4q}}\leq C\br{\expect{(U-V)^{2}}}^{2q}\leq C_H\Dt^{4qH},
\end{equation*}
from which the result follows.
\end{proof}

\begin{theorem}
\label{thm: Z convergence}
Assume $a(t,x)$ satisfies Assumptions~\ref{ass:a_existence} and \ref{ass:a_numerics}. Let $Z(t)$ be the solution to \textup{\eqref{eq: Z differential}} and let $\tZT(t)$ be as in \eqref{def:ZJtilde}.  Let $\tZ_n^n$ be given by \textup{\eqref{eq: Z numerical scheme}}. Then there exists a constant $C_H>0$ such that
    \begin{equation*}
        \br{\E\left[\br{\tZ_n^n-\tZT(t_n)}^2\right]}^\frac{1}{2}\leq 
C_H \Dt^{\min(H,\zeta)}.
    \end{equation*}
\end{theorem}

\begin{proof}

For ease of notation for this proof we suppress dependence on $T$ and $n$, so that $\tJ=\tJT$, $\tZ=\tZT$, $\tJ_i=\tJ_i^n$ and $\tZ_i=\tZ_i^n$.

Since
\begin{equation*}
    \tZ_n=x_0+\Dt\sum_{i=0}^{n-1}\tJ(t_i)a(t_i,\tJ(t_i)^{-1}\tZ_i)=x_0+\intsum \tJ(t_i)a(t_i,\tJ(t_i)^{-1}\tZ_i)ds,
\end{equation*}
by \eqref{eq: Z numerical scheme} and
\begin{equation*}
    \tZ(t_n)=x_0+\int_0^{t_n}\tJ(s)a(s,\tJ(s)^{-1}\tZ(s))ds=x_0+\intsum \tJ(s)a(s,\tJ(s)^{-1}\tZ(s))ds,
\end{equation*}
by \eqref{eq: tilde Z differential}, we can apply the triangle inequality twice, square and take expectation to get 
\begin{equation}
\label{eq:ZI2}
    \E \left[(\tZ_n-\tZ(t_n))^2\right]\leq 3\E \left[I_1^2\right]+3\E\left[I_2^2\right]+3\E\left[I_3^2\right], 
\end{equation}
where
\begin{align*}
    I_1\coloneqq&\left|\Dt\sum_{i=0}^{n-1}\tJ(t_i)a(t_i,\tJ(t_i)^{-1}\tZ_i)-\Dt\sum_{i=0}^{n-1}\tJ(t_i)a(t_i,\tJ(t_i)^{-1}\tZ(t_i))\right|,\\
    I_2\coloneqq&\left|\Dt\sum_{i=0}^{n-1}\tJ(t_i)a(t_i,\tJ(t_i)^{-1}\tZ(t_i))-\intsum \tJ(s)a(s,\tJ(s)^{-1}\tZ(t_i))ds\right|\\
    I_3\coloneqq&\left|\intsum \tJ(s)a(s,\tJ(s)^{-1}\tZ(t_i))ds-\intsum \tJ(s)a(s,\tJ(s)^{-1}\tZ(s))ds\right|.
\end{align*}
Using condition~\ref{enum: cond2} from Assumption~\ref{ass:a_existence} we see that 
\begin{equation*}
    I_1\leq \Dt\sum_{i=0}^{n-1}\tJ(t_i)\left|a(t_i,\tJ(t_i)^{-1}\tZ_i)-a(t_i,\tJ(t_i)^{-1}\tZ(t_i))\right|\leq C\Dt\sum_{i=0}^{n-1}|\tZ_i-\tZ(t_i)|.
\end{equation*}
So, by Jensens's inequality
\begin{equation}  \label{eq:bound_for_I1}
    \E \left[I_1 ^2 \right] \leq CT\Dt\sum_{i=0}^{n-1} \E \left[(\tZ_i-\tilde Z(t_i))^2 \right].
\end{equation}
For $I_2$, adding and subtracting the terms $\tJ(s)a(t_i,\tJ(t_i)^{-1}\tZ(t_i))$ and $\tJ(s)a(s,\tJ(t_i)^{-1}\tZ(t_i))$ we get 
\begin{align*}
 \bigg|  \tJ(t_i)a(t_i,\tJ(t_i)^{-1}\tZ(t_i))   - &\tJ(s)a(s,\tJ(s)^{-1}\tZ(t_i))\bigg| \leq  \left|\tJ(t_i)-\tJ(s)\right| 
 \left|a(t_i,\tJ(t_i)^{-1}\tZ(t_i)) \right|
\\ 
&+\left| \tJ(s) \right| \left|a(t_i,\tJ(t_i)^{-1}\tZ(t_i))  - a(s,\tJ(t_i)^{-1}\tZ(t_i)) \right| \\
& + \left| \tJ(s)\right| \left|a(s,\tJ(t_i)^{-1}\tZ(t_i))-a(s,\tJ(s)^{-1}\tZ(t_i)) \right|.
\end{align*}
Squaring, taking the expectation we find by Assumption~\ref{ass:a_numerics} and the stochastic Cauchy--Schwarz inequality that 
\begin{align*}
 \E \bigg[\bigg( \tJ(t_i)a(t_i,\tJ(t_i)^{-1}\tZ(t_i))  & -  \tJ(s)a(s,\tJ(s)^{-1}\tZ(t_i))\bigg)^2\bigg] \\
& \leq  \br{\E \left[\left(\tJ(t_i)-\tJ(s)\right)^4\right]}^\frac{1}{2} \br{\E \left[a(t_i,\tJ(t_i)^{-1}\tZ(t_i)) ^4\right] }^{\frac{1}{2}}\\
& \quad + C \br{\E\left[\tJ(s) ^4\right]}^\frac{1}{2} \br{\E \left[ \left(1+\left| \tJ(t_i)^{-1} \tZ(t_i)  \right|\right)^4\right]}^\frac{1}{2} \left|s - t_i \right|^ {2 \zeta}\\
& \quad + \br{\E \left[ \tJ(s)^4\right]}^\frac{1}{2}\br{\E \left[\left(a(s,\tJ(t_i)^{-1}\tZ(t_i))
-a(s,\tJ(s)^{-1}\tZ(t_i)) \right)^4\right]}^\frac{1}{2}.
\end{align*}
By Assumption~\ref{ass:a_existence} (i) and (ii) and Lemmas~\ref{lem:J_BoundedMoments} and \ref{lemma: DJorder} 
we see 
\begin{equation}  \label{eq:bound_for_I2}
    \E \left[I_2^2 \right] \leq C_HT\Dt^{\min(2H,2\zeta)}.
\end{equation}
For $I_3$ we have
\begin{equation*}
    I_3\leq \intsum \tJ(s)\left| a(s,\tJ(s)^{-1}\tZ(t_i))-a(s,\tJ(s)^{-1}\tZ(s))\right|ds\leq C\intsum |\tZ(t_i)-\tZ(s)|ds,
\end{equation*}
by the Lipschitz assumption on $a$.
Lemma~\ref{lem:Zcontrol} shows that 
\begin{equation} \label{eq:bound_for_I3}
    \E \left[I_3 ^2 \right] \leq C\intsum \E\left[\br{\tZ(t_i)-\tilde Z(s)}^2\right]ds \leq C_H\Dt^2.
\end{equation}
Bringing inequalities \eqref{eq:bound_for_I1}, \eqref{eq:bound_for_I2} and \eqref{eq:bound_for_I3} together we have 
\begin{align*}
\label{eq: Zn square sum}
\E\left[(\tZ_n-\tilde Z(t_n))^2\right] 
& \leq CT\Dt \sum_{i=0}^{n-1} \E \left[|\tZ_i-\tilde Z(t_i)|^2 \right]  +C_H\left(\Dt ^{\min(2H,2\zeta)} +\Dt ^2 \right).
\end{align*}
Thus by the discrete Gronwall inequality, found for instance in \cite{Clark1987}, we have the result.
\end{proof}

\subsection*{Proof of Theorem~\ref{thm: RMSE}} 
This follows from Theorem~\ref{thm: Z convergence}. By \eqref{eq: X(t) expression} and \eqref{eq: Xn num method} we have
    \begin{equation*}
        X_n-X(t_n)=\tJT(t_n)^{-1}\br{\tZ_n^n-\tZT(t_n)}.
    \end{equation*}
    Applying the Cauchy--Schwarz inequality gives
    \begin{equation*}
        \E\left[\br{X_n-X(t_n)}^2\right]\leq \br{\E\left[\br{\tJT(t_n)^{-1}}^2\right]}^\frac{1}{2}\br{\E\left[\br{\tZ^n_n-\tZT(t_n)}^2\right]}^\frac{1}{2}
    \end{equation*}
    Since $\E\left[\br{\tJT(t_n)^{-1}}^2\right]\leq C_H$ by Lemma~\ref{lem:J_BoundedMoments}, the result follows directly from Theorem~\ref{thm: Z convergence}. 

\begin{remark}
\label{remark: C_H behavior}
    The explicit dependence on $H$ of the constant $C_H$ appearing in Theorem \ref{thm: RMSE} can be tracked and is given by
    \begin{equation*}
        C_H= 
        \begin{cases}
         C\exp(cT^{2H}),& \text{if } H\leq \frac{1}{2},\\
         CHT^{2H-1}\exp(cT^{2H}),& \text{if } H> \frac{1}{2}.
        \end{cases}  
    \end{equation*}
    for some constants $C,c>0$ not depending on $H$. In particular, this constant does not degenerate for small values of $H$.
\end{remark}

\section{Numerical Results}\label{sec: Numerics}

In this section we illustrate our numerical method \GBMEM and examine paths for different $H$ and compare it to \MISHURA as well as two other methods we propose.
We then examine convergence for fixed $H$ and finally examine the rates of convergence as we vary $H$. We present results for the autonomous case $a(x)$ only. Code to generate the figures below are available from \url{https://github.com/Gabriel-Lord/WIS-SDEs-and-fBM.git}.

\subsection{Numerical methods}

We compare \GBMEM given in \eqref{eq: Z numerical scheme} and \eqref{eq: Xn num method} (see 
Algorithms~\ref{alg: single time} and \ref{alg: path sample}) to the method of \cite{Mishura2008}, \MISHURA. As noted in Remark~\ref{rem:Mishura1}, the two methods coincide if $\alpha=0$, and in the case of gfBm \eqref{eq: fGBM SDE}, where $a(t,x)=0$, the method \GBMEM is exact, whereas \MISHURA is not. 
We now propose two additional methods for solving \eqref{eq: tilde Z differential}, which also provide exact solutions to \eqref{eq: fGBM SDE}.

\subsubsection*{The exponential freeze method}
Consider the differential equation \eqref{eq: tilde Z differential} for the process $\tZT(t)$ and let $g(t,x)\coloneqq a(t,x)/x$ for $x\neq0$. In this case, we have
\begin{equation*}
    \frac{d}{dt}\tZT(t)=\tZT(t)g(t,\tJT(t)^{-1}\tZT(t)).
\end{equation*}
If we consider this differential equation on a small interval $[t_i,t_{i+1}]$, we can reasonably freeze $g(t,\tJT(t)^{-1}\tZT(t))$ at the constant value $g(t_i,\tJT(t_i)^{-1}\tZT(t_i))$. This means that on $[t_i,t_{i+1}]$ the process $\tZT(t)$ approximately satisfies 
\begin{equation}
\label{eq: freeze derivative}
    \frac{d}{dt}\tZT(t)\approx \tZT(t)g(t_i,\tJT(t_i)^{-1}\tZT(t_i)).
\end{equation}
With this approximation, our numerical scheme for $\tZT(t)$ becomes
\begin{equation*}
    \tZ^n_{i+1}=\tZ^n_i\exp\br{\Delta tg(t_i,\tZ^n_i/\tJ^n_i)}.
\end{equation*}
When $a(t,x)=0$ we have $g(t,x)=0$ as well and so this method solves the gfBm exactly. We denote this method \EXPFREEZE.

\subsubsection*{A Rosenbrock type approximation}
Alternatively, we could consider some specific value $A_i\in\R$ and decompose the differential equation \eqref{eq: tilde Z differential} for $\tZT(t)$ on $[t_i,t_{i+1}]$ as
\begin{equation}
\label{eq: Rosenbrock start}
    \frac{d}{dt}\tZT(t)=A_i\tZT(t)+\left[\tJT(t)a(t,\tJT(t)^{-1}\tZT(t))-A_i\tZT(t)\right].
\end{equation}
By a Lie--Trotter splitting we have 
\begin{equation*}
    \tZ^n_{i+1}=\tZ^n_i\exp\br{\Dt A_i}+\Dt\left[\tJ^n_ia(t_i,\tZ^n_i/\tJ^n_i)-A_i\tZ^n_i\right].
\end{equation*}
We still have to make a choice of $A_i$. If $a(t,x)$ is differentiable in its second component with derivative $\frac{\partial}{\partial x}a(t,x)$, then one possibility is to choose
\begin{equation*}
    A_i=\frac{\partial}{\partial x}a(t_i,\tZ^n_i/\tJ^n_i).
\end{equation*}
Provided that $\Delta t=t_{i+1}-t_i$ is small enough, this approximately cancels the first derivative term of the second component in the Taylor expansion of $\tJT(t)a(t,\tJT(t)^{-1}\tZT(t))$ around $t_i$ in \eqref{eq: Rosenbrock start}. 

Our resulting numerical scheme, denoted \Rosenbrock, is 
\begin{equation}
\label{eq: Rosenbrock scheme}
    \tZ^n_{i+1}=\tZ^n_i\exp\br{\Dt A_i}+\Dt\left[\tJ^n_ia(t_i,\tZ^n_i/\tJ^n_i)-A_i\tZ^n_i\right],\qquad 
    A_i=\frac{\partial}{\partial x}a(t_i,\tZ^n_i/\tJ^n_i).
\end{equation}
When $a(t,x)=0$, we see that $A_i=0$ as well and we once again have $\tZ^n_i=x_0$ for all values of $i$. So this numerical method also solves the gfBm exactly.

\begin{figure}[!h]
     \centering
     (a) \hspace{0.45\textwidth} (b) \\
     \includegraphics[width=0.45\textwidth]{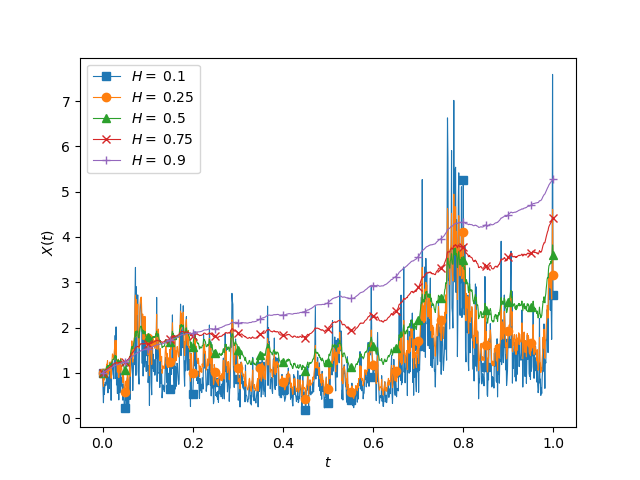}
     \includegraphics[width=0.45\textwidth]{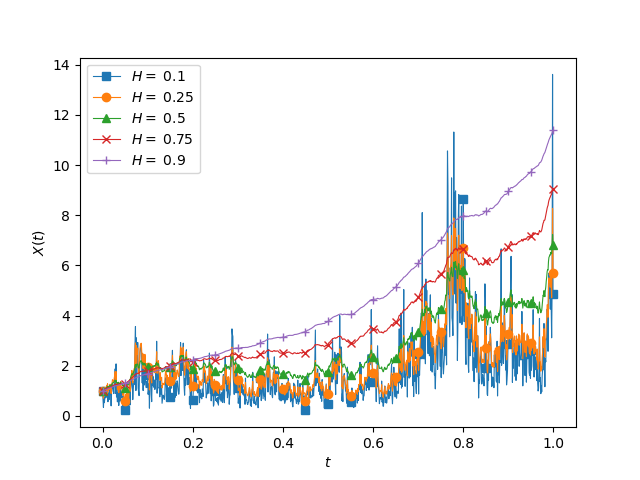}
     \caption{Sample paths of \GBMEM for different values of $H$ using same random numbers for \eqref{eq: quasi-linear SDE} with $\beta=1$, $a(x)=\frac{4x}{1+x^2}$, $x_0=1$, $\Dt=0.001$ and 
     (a) $\alpha=0$ (\MISHURA) 
     (b) $\alpha=1$ (\GBMEM).}
     \label{fig:1}
\end{figure}

\subsection{Numerical Results}
In Fig.~\ref{fig:1} we plot sample paths for $H\in\{0.1,0.25,0.5,0.75,0.9\}$ on the interval $[0,T]$ for \eqref{eq: quasi-linear SDE} with 
\begin{equation}
\label{eq: num ex}
\alpha\in\{0,1\}, \quad \beta=1, \quad a(x)=\frac{4x}{1+x^2}, \quad x_0=1, \quad T=1,
\end{equation}
and taking $\Dt=0.001$. In Fig.~\ref{fig:1}
(a) $\alpha=0$ so \GBMEM and \MISHURA coincide and in 
(b) $\alpha=1$ we use \GBMEM. 
The random numbers used are the same for each path so that the effect of $H$ on for example the regularity of the paths is evident.

In Fig.~\ref{fig:2} we compare the same sample paths for the four different methods for $H=0.25$ in (a) and $H=0.75$ in (b).  Here we take \eqref{eq: quasi-linear SDE} with 
$\alpha=1$ and other parameters as in \eqref{eq: num ex} with $\Dt=0.001$.
We see good agreement between all four methods, indeed small differences can only been observed when zooming in on the paths.

\begin{figure}[!h]
     \centering
     (a) \hspace{0.45\textwidth} (b) \\
     \includegraphics[width=0.45\textwidth]{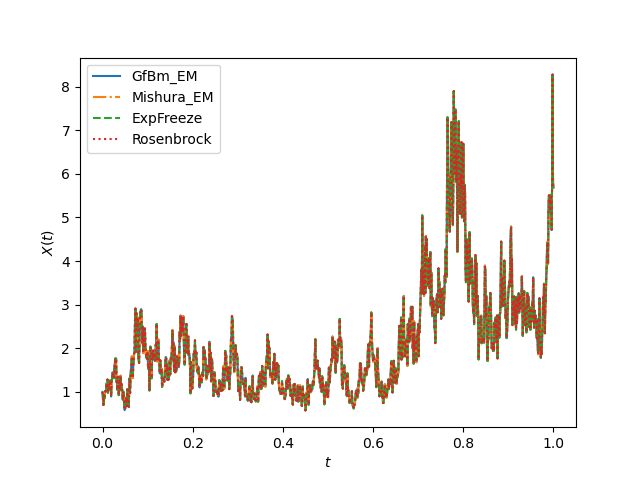}
     \includegraphics[width=0.45\textwidth]{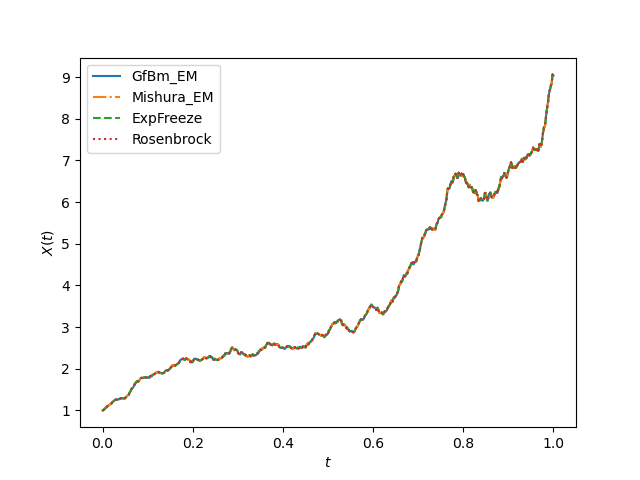}
     \caption{Sample paths of the four methods with $\alpha=1$, $\beta=1$, $a(x)=\frac{4x}{1+x^2}$, $x_0=1$, $\Dt=0.001$ using same random numbers with (a) $H=0.25$, (b) $H=0.75$.
     }
     \label{fig:2}
\end{figure}

We now examine in Fig.~\ref{fig:3} numerically the convergence of our method \GBMEM for the same problem of Fig.~\ref{fig:2} given in \eqref{eq: num ex} with $\alpha=1$. The root mean square error (RMSE) is estimated via Monte Carlo estimation with a sample size of $500$. As reference solutions we have taken \GBMEM approximations with $n=2^{19}$ steps. For both $H=0.25$ and $H=0.75$ we see that all the methods have approximately the same RMSE and rate of convergence. Note that the estimated rate of convergence in both cases is considerably better than the theoretical result of Theorem~\ref{thm: RMSE}. For $H=0.25$ we observe a rate of $0.771\approx H+\frac{1}{2}$ and for $H=0.75$ we observe a rate of $1.009\approx 1$.
\begin{figure}[!h]
     \centering
     (a) \hspace{0.45\textwidth} (b) \\
     \includegraphics[width=0.45\textwidth]{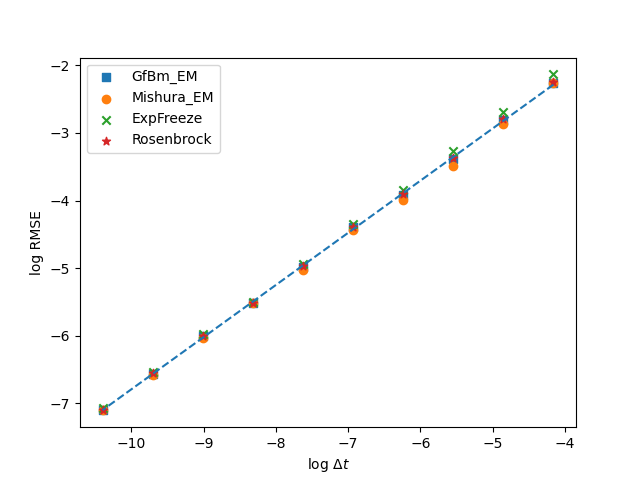}
     \includegraphics[width=0.45\textwidth]{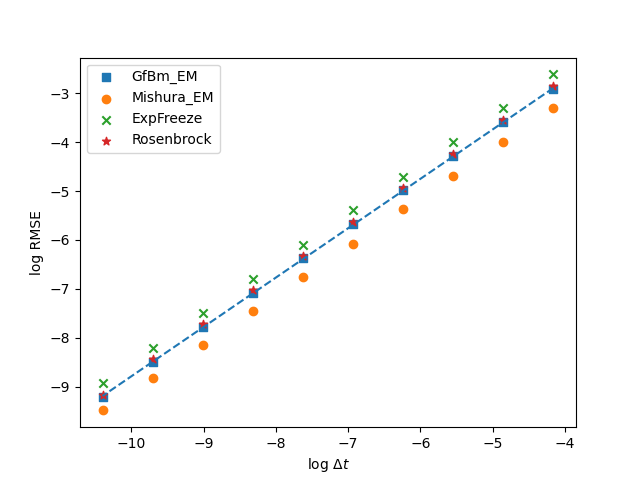}
     \caption{The estimated RMSE of the four methods for various values of $\Dt$ for the quasi-linear SDE with parameters in \eqref{eq: num ex}, $\alpha=1$ and $T=1$, including a linear fit for the \GBMEM method. We have (a) $H=0.25$, linear fit slope $0.771$ (b) $H=0.75$, linear fit slope $1.009$.
     }
     \label{fig:3}
\end{figure}

In Figs.~\ref{fig:4}--\ref{fig:6} we examine the rates of convergence against $H$ as estimated from slopes such as in Fig.~\ref{fig:3}. In order to quantify the uncertainty, we divide the Monte Carlo sample of the RMSE for each value of $H$ into batches. Among each batch the rate of convergence is estimated. We then plot the mean of the estimated rates of convergence along with error bars of the size of the sample standard deviation among the batches.
In Fig.~\ref{fig:4} we examine \eqref{eq: quasi-linear SDE} with the parameters from \eqref{eq: num ex} and consider the method \GBMEM (a) with $\alpha=0$ (and hence it coincides with \MISHURA) and in Fig.~\ref{fig:4} (b) with $\alpha =1$. We indicate in our plots the theoretical rate of convergence $H$ from Theorem~\ref{thm: RMSE}, the estimated rate of convergence, the linear fit for $H<\frac{1}{2}$ and the conjectured result of \eqref{eq:conjecture}.
In Fig.~\ref{fig:5} (a) we keep the parameters as in \eqref{eq: num ex} but increase the noise intensity to $\beta=2$. The observed rates are robust to the noise intensity (see also Fig.~\ref{fig:6} (a)). 
In Fig.~\ref{fig:5} (b) we examine a different SDE and take 
$a(x)=\cos(x)$, $\alpha=-1$, $\beta=0.5$, $x_0=10$ and solve to $T=1$. Similarly in Fig.~\ref{fig:6} (a) and  (b) we take $a(x)=25\log(1+x^2)$, $\alpha=0$, $\beta=5$, $x_0=25$ and solve to $T=1$. The total Monte Carlo sample size is now taken as $2500$, divided in batches of $250$ to obtain the given error bars. In Fig.~\ref{fig:6} (a) we plot the rate of convergence against $H$ and in (b) we examine $\log(C_H)$ where $C_H$ is the error constant from the error estimation. We observe in (a) the conjectured convergence rate \eqref{eq:conjecture} whereas in (b) it is clear that the constant $C_H$ does not degenerate for small $H$, which is in line with Remark \ref{remark: C_H behavior}. 

We have also performed further experiments with different functions $a(x)$ (including non-globally Lipschitz functions), solving over longer time intervals $T$ and larger noise intensities. These and the simulations presented here in Figs.~\ref{fig:4}--\ref{fig:6} all support the conjecture in \eqref{eq:conjecture}.

\begin{figure}[!h]
     \centering
          (a) \hspace{0.48\textwidth} (b) \\
     \includegraphics[width=0.48\textwidth]{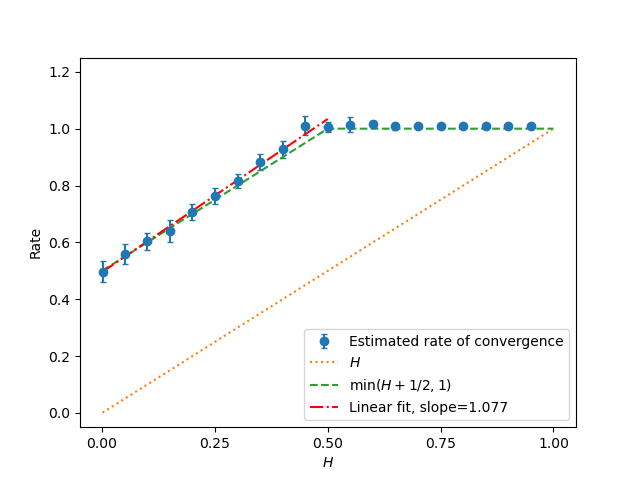}
     \includegraphics[width=0.48\textwidth]{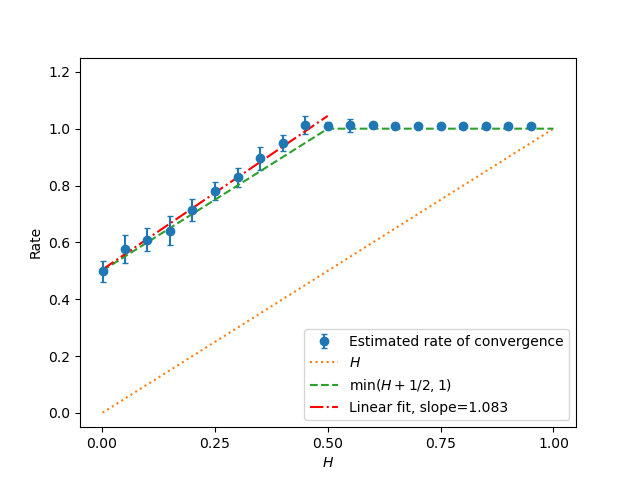}
     \caption{
     Estimated rates of convergence for $\beta=1$, $a(x)=\frac{4x}{1+x^2}$, $x_0=1$ and $T=1$.  The total Monte Carlo sample size is $500$, divided in batches of $50$ to obtain the given error bars. In (a) $\alpha=0$ (\MISHURA) and in (b) $\alpha=1$ (\GBMEM).}
     \label{fig:4}
\end{figure}

\begin{figure}[!h]
     \centering
          (a) \hspace{0.48\textwidth} (b) \\  
     \includegraphics[width=0.48\textwidth]{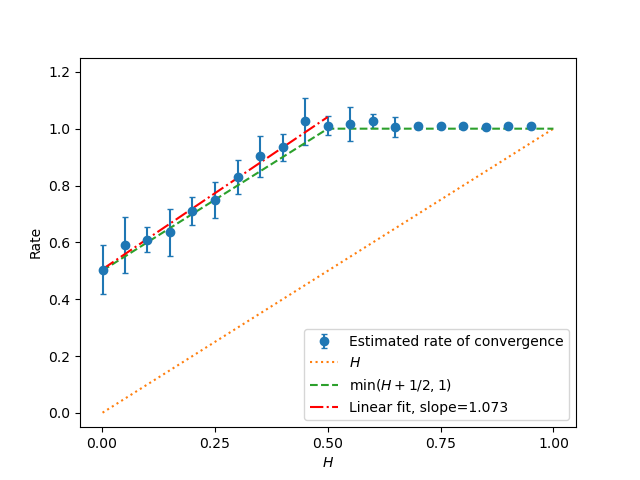}
          \includegraphics[width=0.48\textwidth]{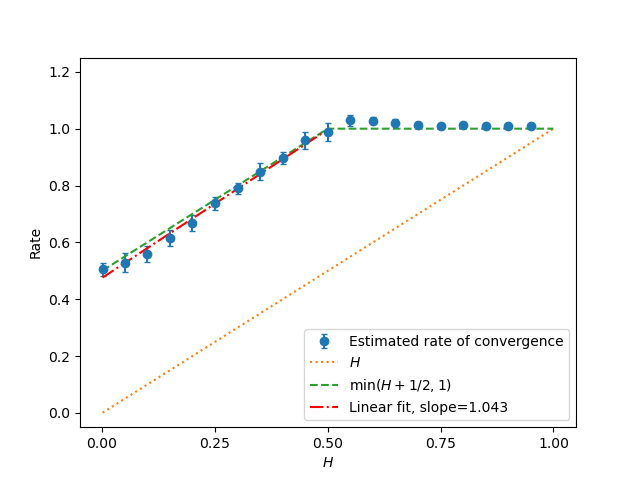}
     \caption{(a) The estimated rates of convergence for $\alpha=1$, $\beta=2$, $a(x)=\frac{4x}{1+x^2}$, $x_0=1$ and $T=1$ (\GBMEM).
     (b) The estimated rates of convergence for $\alpha=-1$, $\beta=0.5$, $a(x)=\cos(x)$, $x_0=10$ and $T=1$ (\GBMEM). In (a) and (b) the total Monte Carlo sample size is $500$, divided in batches of $50$ to obtain the given error bars.
     }
     \label{fig:5}
\end{figure}

\begin{figure}[!h]
     \centering
          (a) \hspace{0.48\textwidth} (b) \\
     \includegraphics[width=0.48\textwidth]{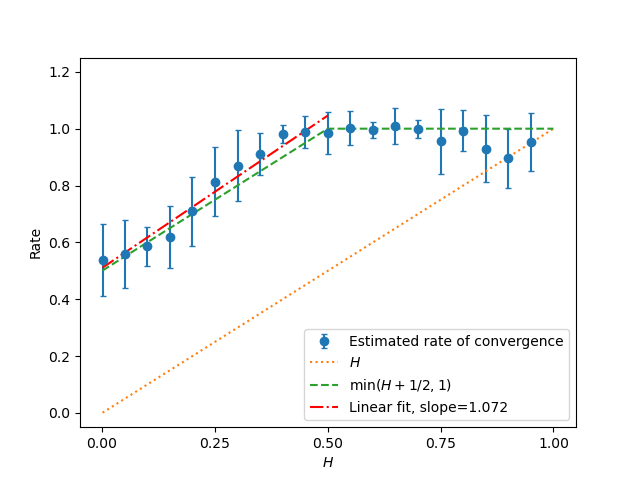}
      \includegraphics[width=0.48\textwidth]{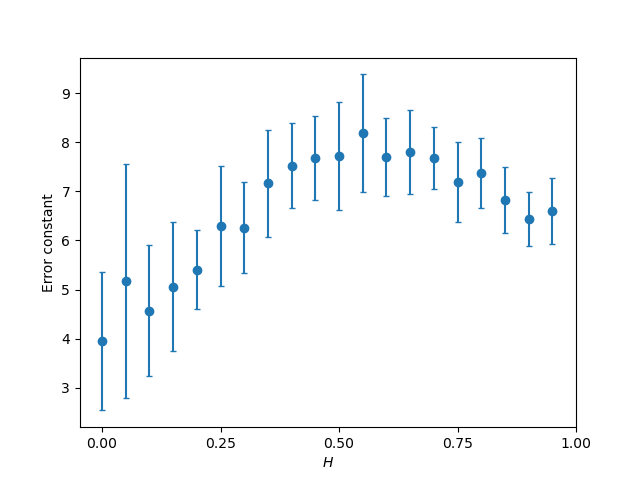}
     \caption{(a) The estimated rates of convergence for $\alpha=0$, $\beta=5$, $a(x)=25\log(1+x^2)$, $x_0=25$ and $T=1$ (\MISHURA). The total Monte Carlo sample size is $2500$, divided in batches of $250$ to obtain the given error bars.   
     (b) The corresponding estimation of the error constant $\log(C_H)$.}
    \label{fig:6}
\end{figure}

\section{Discussion and Conclusion}\label{sec:discussion}

One open question is whether the theoretical rate of convergence can be improved and approach the conjectured estimate \eqref{eq:conjecture} observed in our numerical experiments.  

In our proof of the rate of convergence it is the term in $I_2$ in \eqref{eq:ZI2} in Theorem~\ref{thm: Z convergence} that the limiting factor in obtaining a higher rate of convergence, provided $\zeta>H$ in Assumption~\ref{ass:a_numerics}. One potential way to show a faster convergence rate is to use 
the fractional It\^o formula of Theorem~\ref{thm: Ito formula}. This is the same approach used in \cite{ErdoganLord} to analyse  $H=\frac{1}{2}$ to obtain rate one convergence, or is used to analyse a standard Milstein type method and obtain higher rates of convergence. The expansion
of $\tJT(s)a(\tJT(s)^{-1}\tZT(t_i))$ leads to three terms that need to be bounded with an appropriate order. For two of the terms this can be achieved. However there is one term of the form
\begin{equation}
 \label{eq:ThetaDef}
    \Theta_i\coloneqq \int_{t_i}^{t_{i+1}}  \int_{t_i}^{s}  \der{\Phi^{(i)}}{x} (t_i,\fbm{t_i}) d\fbm{r}  ds, 
\end{equation} 
where
$$ \der{ \Phi^{(i)}} {x} (r,\fbm{r})
 = -\beta \left(a\left(\tJT(r)^{-1}\tZT(t_i)\right) \tJT(r) -
a'\left(\tJT(r)^{-1}\tZT(t_i)\right) \tZT(t_i)\right).$$
We need to examine 
$\E\left[ \left| \sum_{i=0}^ {n-1}\Theta_i  \right|^2\right]$. 
Although for all $H>\frac{1}{4}$ we can prove there exists a constant $C_H>0$ such that 
$$\E\left[ \left| \sum_{i=0}^ {n-1}\Theta_i ^2  \right|\right] \leq C_H \Delta t^ {2H+1},$$
the issue is with cross terms $\E\left[ \Theta_i \Theta_k  \right]$. When $H=\frac{1}{2}$ we can prove for $k>i$ that $\E\left[ \Theta_i \Theta_k  \right]=0$, however, for $H\neq \frac{1}{2}$ this is no longer true, as we no longer have a martingale.
When 
$\der{\Phi^{(i)}}{x} (t_i,\fbm{t_i}) $ is deterministic then we can show the cross terms give a term of higher order, in fact of $O(\Dt^{\min(2H+1,2)})$. However this higher order is lost when the stochastic nature of $\der{\Phi^{(i)}}{x} (t_i,\fbm{t_i})$ is taken into account by an application of the stochastic Cauchy--Schwarz inequality. After looking in detail at the expansion we suspect an alternative approach is required. We note that in the case that $a(t,x)$ is linear in $x$ then the problematic term is zero. 
One possible approach to deal with the terms $\Theta_i$ is to avoid the application of the stochastic Cauchy--Schwarz inequality and instead try to use Malliavin calculus, together with the assumption of some higher order derivative bounds on $a$. 
We leave this for future research.

Both the \MISHURA and the \GBMEM methods are obtained by using an explicit Euler approximation of \eqref{eq: integral Z}. In Section \ref{sec: Numerics} we suggest two alternative ways to approximate the solution to the underlying differential equation. Nonetheless, there are of course a lot of other (higher order) methods available for the approximation of the integral in \eqref{eq: integral Z}. Exploration of these methods 
would be interesting and could possibly lead to higher order schemes.

The quasilinear SDE \eqref{eq: quasi-linear SDE} we consider in this paper is one-dimensional. We believe that the results presented in this paper can be extended to a multi-dimensional SDE with a single one-dimensional driving fBm $\fbm{t}$ of the form
\begin{equation*}
dX(t)=\mathbf{A} X(t)dt+a(t,X(t))dt+\mathbf{B} X(t)dB^H(t), \quad X(0)=x_0, \quad t\in[0,T],
\end{equation*}
where $\mathbf{A},\mathbf{B}\in\R^{d\times d}$ are two commuting matrices, $x_0\in\R^d$, $T>0$ and $a\colon[0,T]\times\R^d\to\R^d$. The WIS integral in this SDE is interpreted component wise. In this setting, the geometric fractionional Brownian motion is the solution to the differential equation in $((\calS)^*)^d$ given by $X(0)=x_0$ and 
\begin{equation}
\label{eq: multi-dimensional GfBm DE}
    \frac{d}{dt}X(t)=\mathbf{A}X(t)dt+\mathbf{B}X(t)\dm W^H(t).
\end{equation}
Theorem 3.1.5 of \cite{Holden1996} shows that the solution to this differential equation is given by
\begin{equation*}
    X(t)=\wexp{t\mathbf{A}+\fbm{t}\mathbf{B}}x_0,
\end{equation*}
where $\textup{exp}^\dm$ now is the matrix Wick exponential. A relation similar to \eqref{eq: wexp to exp} can be derived for this setting, which will show that this solution can be written as
\begin{equation*}
    X(t)=\exp\br{t\mathbf{A}-\frac{1}{2}t^{2H}\mathbf{B}^2+\fbm{t}\mathbf{B}}x_0.
\end{equation*}
The fact that $X(t)$ solves \eqref{eq: multi-dimensional GfBm DE} can also be verified using the fractional It\^o formula given in Theorem \eqref{thm: Ito formula}. With this observation in mind, we believe a similar approach to that of Section \ref{sec: quasilinear} can be taken.


Extending the quasilinear SDE even further to allow for multi-dimensional driving fractional noise requires more effort. First, we will need a multi-dimensional version of WIS integration, which can be found in in Section 4.6 of \cite{Biagini2008}. This definition allows for distinct values of the Hurst parameters $H$ for each of the one-dimensional fBms involved. With this notion of multi-dimensional WIS integration, one needs to find the resulting multi-dimensional geometric fBm and use it to construct a suitable integrating factor $J(t)$. This integrating factor will consist of a Wick exponential, which, for the construction of a numerical method, needs to be rewritten into a regular exponential using a relation similar to that of \eqref{eq: wexp to exp}. In this setting with multiple driving fBms, deriving such a relation might prove to be challenging. A crucial ingredient of the subsequent existence and uniqueness proof and the construction of a numerical method is Gjessing's lemma. It is therefore an interesting question whether this lemma extends to the setting of multi-dimensional WIS integration given in \cite[Section 4.6]{Biagini2008} and whether the integrating factor is of a form that allows for the application of this lemma.

It would also be interesting to extend the setting discussed in this paper to that of non-linear diffusion terms. The transformation to the process $Z$ in the proof of Theorem \ref{thm: quasilinear SDE} is essentially a Doss--Sussmann transformation. This observation might help to extend the results to more general diffusion terms. In such more general settings we unfortunately run into the fact that the involved Wick exponential cannot be easily rewritten using \eqref{eq: wexp to exp}, making the construction of a numerical method challenging. One could also consider a Lamperti transformation to extend the results to more general diffusion terms. This approach will again suffer from the appearance of Wick products that cannot be easily calculated, and you can potentially lose the Lipschitz continuity of the drift after the transform, requiring a new proof of existence and uniqueness.

We can show the difficulty in the following particular setting, where we to transform the SDE to one with a linear diffusion term, at the cost of needing to evaluate Wick products.

For $a\colon [0,T]\times\R\to\R$, $\beta\in\R$ and $\theta>0$ consider the SDE in $(\calS)^*$ given by
\begin{equation*}
    \frac{d}{dt}X(t)=a(t,X(t))+\beta X(t)^{\dm\theta}\dm W^H(t).
\end{equation*}
See Example 2.6.15 in \cite{Holden1996} for some remarks on $X^{\dm\theta}$ and the fact that $X^{\dm\theta}\dm X^{\dm(-\theta)}$=1. All we need for this to be well-defined is that $\E[X]\neq0$. 

Let $Y(t)\coloneqq \wexp{X(t)^{\dm(1-\theta)}}$. Note that if $X(0)=x_0$ is deterministic, we have $Y(0)=\wexp{x_0^{\dm(1-\theta)}}=\exp(x_0^{1-\theta})$. Applying the Wick chain rule two times gives
\begin{align*}
    \frac{d}{dt}Y(t)&=(1-\theta)Y(t)\dm X(t)^{\dm(-\theta)}\dm\frac{d}{dt}X(t)\\
    &=(1-\theta)Y(t)\dm X(t)^{\dm(-\theta)}\dm a(t,X(t))+\beta(1-\theta) Y(t)\dm W^H(t).
\end{align*}
In particular, if $a(t,x)=\alpha x$ for some $\alpha\in\R$, we have 
\begin{align*}
    \frac{d}{dt}Y(t)&=\alpha(1-\theta)Y(t)\dm X(t)^{\dm(1-\theta)}+\beta(1-\theta) Y(t)\dm W^H(t)\\
    &=\alpha(1-\theta) Y(t)\dm \log^\dm(Y(t))+\beta(1-\theta) Y(t)\dm W^H(t).
\end{align*}
where we, once again, refer to Example 2.6.15 in \cite{Holden1996} for some words on $\log^\dm(X)$. Hence, we could use the numerical methods suggested in this paper in order to sample $Y(t)$ if we know how to express the function $Z\mapsto Z\dm\log^\dm(Z)$ as a function $f\colon\R\to\R$.

We could also consider $a\colon(\calS)^*\to(\calS)^*$ with the specific choice of $a(X)=\alpha X^{\dm\theta}$. In this case we get the SDE
\begin{equation*}
    \frac{d}{dt}Y(t)=\alpha(1-\theta)Y(t)+\beta(1-\theta) Y(t)\dm W^H(t).
\end{equation*}
The solution to this equation is given by the geometric fBm:
\begin{align*}
    Y(t)&=Y(0)\dm\wexp{\alpha(1-\theta)t+\beta(1-\theta)\fbm{t}}\\
    &=\wexp{x_0^{1-\theta}+\alpha(1-\theta)t+\beta(1-\theta)\fbm{t}}.
\end{align*}
Hence, if $\theta\neq 1$, we have
\begin{equation*}
    X(t)=(\log^\dm(Y(t)))^{\dm\br{\frac{1}{1-\theta}}}=(x_0^{1-\theta}+\alpha(1-\theta)t+\beta(1-\theta)\fbm{t})^{\dm\br{\frac{1}{1-\theta}}},
\end{equation*}
a result that also can be obtained via the substitution $Y(t)=X^{\dm 1-\theta}$.
We can deal with the Wick powers in the special case $\theta=\frac{n-1}{n}$ for some $n\in\N$, $n\neq 1$, which corresponds to $\frac{1}{1-\theta}=n$. The solution is then explicitly given by
\begin{align*}
    X(t)&=\sum_{k=0}^n\binom{n}{k}\br{x_0^{1/n}+\frac{\alpha t}{n}}^{n-k}\br{\frac{\beta}{n}}^k\fbm{t}^{\dm k}\\
    &=\sum_{k=0}^n\binom{n}{k}\br{x_0^{1/n}+\frac{\alpha t}{n}}^{n-k}\br{\frac{\beta}{n}}^kt^{kH}h_k\br{\frac{\fbm{t}}{t^H}},
\end{align*}
which can be checked using several results mentioned in \cite{Holden1996}.

In conclusion, we have shown how to implement and proved convergence for all $H\in(0,1)$ for a numerical method for the SDE \eqref{eq: quasi-linear SDE} where the stochastic integral is understood in the WIS sense. However the proof of conjecture \eqref{eq:conjecture}, the extension to a system of WIS SDEs and to non-linear noise terms are interesting future challenges. 
We have not discussed or examined numerically weak convergence and in particular what rate is observed. This would make an interesting follow up to this work. Although our analysis is restricted to one-dimension we believe it should be possible to extend to multi-dimensions.
This work opens up the possibility to efficiently simulate SDEs for all values of $H$ and hence the ability to examine and simulate fBm models enabling their uptake.


\appendix

\section{The Hida spaces}\label{sec:Hida}

To accommodate the derivative of factional Brownian motion $B^H$, we introduce the Hida spaces. We refer to \cite[Section 2.3]{Holden1996} for more details.  

Suppose $F=\sum_{\alpha\in\J}c_\alpha\calH_\alpha$. For any $q\in\R$ we define
\begin{equation*}
    \|F\|_q\coloneqq\br{\sum_{\alpha\in\J}c_\alpha^2\alpha!(2\N)^{q\alpha}}^{\frac{1}{2}},
\end{equation*}
where $(2\N)^\gamma\coloneqq\prod_{i=1}^n(2i)^{\gamma_i}$ for $\gamma\in\R^n$. Note that this expression does not need to be finite. 

\begin{definition}
The \emph{Hida test function space}, denoted by $(\calS)$, consists of all $\varphi\in{L^2(\P)}$ such that 
\begin{equation*}
    \|\varphi\|_k<\infty, \qquad \textup{for all } k\in\N.
\end{equation*}
The \emph{Hida distribution space}, denoted by $(\calS)^*$, consists of all expansions $F=\sum_{\alpha\in\J}c_\alpha\calH_\alpha$
such that 
\begin{equation*}
    \|F\|_{-q}<\infty, \qquad \textup{for some } q\in\N.
\end{equation*}
\end{definition}

By definition, $(\calS)\subset{L^2(\P)}$. Moreover, if $F=\sum_{\alpha\in\J} c_\alpha\calH_\alpha\in{L^2(\P)}$, then we see that
    \begin{equation*}
        \|F\|_0^2=\sum_{\alpha\in\J} c_\alpha^2 \alpha!=\E[F^2]<\infty,
    \end{equation*}
    which implies that $F\in(\calS)^*$. As a result, we have the inclusions
    \begin{equation*}
        (\calS)\subset{L^2(\P)}\subset(\calS)^*.
    \end{equation*}
    Notice in particular the analogy with the inclusions of the Schwartz space and its dual, which are $\calS(\R)\subset{L^2(\R)}\subset\calS'(\R)$. 

\begin{remark}
    More generally, it can be shown that for any $p>1$, we have that
    \begin{equation*}
        (\calS)\subset L^p(\P)\subset(\calS)^*.
    \end{equation*}
    See Corollary 2.3.8 in \cite{Holden1996} for a proof.
\end{remark}

\begin{remark}
    It can be shown that both $(\calS)$ and $(\calS)^*$ are vector spaces. Some simple calculations show that for every $k\in\N$, the function $\|\cdot\|_k$ is a norm on $(\calS)$. 
\end{remark}

We equip $(\calS)^*$ with the final topology generated by $\|\cdot\|_{-q}$, with $q\in\N$. This topology is such that convergence holds if and only if convergence holds in $\|\cdot\|_{-q}$ for some $q\in\N$. Likewise, we equip $(\calS)$ with the topology generated by the norms $\|\cdot\|_{k}$, with $k\in\N$. This topology is such that convergence holds if and only if convergence holds in $\|\cdot\|_k$ for all $k\in\N$. As the notation suggests, $(\calS)^*$ can be interpreted as the dual space of $(\calS)$.

\section*{Acknowledgments}
This project originated from a research visit by the second author to the University of Oslo, funded by 
STORM: Stochastics
for Time-Space Risk Models, from the Research Council of Norway (RCN), project number: 274410.

\section*{Funding}
The first author is supported by Eski\c{s}ehir Technical University Scientific Research Project Commission under grant no: 23ADP097. The third author is supported by NWO grant VI.Vidi.213.070.

\printbibliography[heading=bibintoc,title={References}]

@article{NeuenkirchNourdin,
    author = {Neuenkirch, A. and Nourdin, I.},
    title = {Exact Rate of Convergence of Some Approximation Schemes Associated to {SDE}s Driven by a Fractional Brownian Motion},
    journal = {Journal of Theoretical Probability},
    year = 2007,
    doi={10.1007/s10959-007-0083-0},
}

@article{10.1093/imanum/drad019,
    AUTHOR = {Ding, Xiao-Li and Wang, Dehua},
     TITLE = {Regularity analysis for {SEE}s with multiplicative f{B}ms and
              strong convergence for a fully discrete scheme},
   JOURNAL = {IMA J. Numer. Anal.},
  FJOURNAL = {IMA Journal of Numerical Analysis},
    VOLUME = {44},
      YEAR = {2024},
    NUMBER = {3},
     PAGES = {1435--1463},
}

@article{AlfonsiKebaier,
    author = {Alfonsi, A. and Kebaier, A.},
    title = {Approximation of Stochastic {Volterra} Equations with kernels of completely monotone type},
    journal = {Mathematics of Computation},
    volume = 93,
    number = 346,
    pages = {643--677},
    year = 2024
}

@article{ErdoganLord,
    author = {U. Erdogan and G. J. Lord},
        Journal = {IMA Journal of Numerical Analysis},
        Month = {03},
        Number = {2},
        Pages = {820-846},
        Title = {A new class of exponential integrators for {SDE}s with multiplicative noise},
        Volume = {39},
        Year = {2018},
}

@book {Biagini2008,
    AUTHOR = {Biagini, Francesca and Hu, Yaozhong and {\O}ksendal, Bernt and
              Zhang, Tusheng},
     TITLE = {Stochastic calculus for fractional {B}rownian motion and
              applications},
    SERIES = {Probability and its Applications (New York)},
 PUBLISHER = {Springer-Verlag London, Ltd., London},
      YEAR = {2008},
   
}

@book {Holden1996,
    AUTHOR = {Holden, Helge and {\O}ksendal, Bernt and Ub{\o}e, Jan and Zhang,
              Tusheng},
     TITLE = {Stochastic partial differential equations},
    SERIES = {Probability and its Applications},
      SUBTITLE = {A modeling, white noise functional approach},
 PUBLISHER = {Birkh\"{a}user Boston, Inc., Boston, MA},
      YEAR = {1996},
   
}

@book {DiNunno2009,
    AUTHOR = {Di Nunno, Giulia and {\O}ksendal, Bernt and Proske, Frank},
   YEAR = {2009},
     TITLE = {Malliavin calculus for {L}\'{e}vy processes with applications to
              finance},
    SERIES = {Universitext},
 PUBLISHER = {Springer-Verlag, Berlin},
   
  
}

@book {Mishura2008,
    AUTHOR = {Mishura, Yuliya S.},
     TITLE = {Stochastic calculus for fractional {B}rownian motion and
              related processes},
    SERIES = {Lecture Notes in Mathematics},
    VOLUME = {1929},
 PUBLISHER = {Springer-Verlag, Berlin},
      YEAR = {2008},
  
}

@article {Elliott2003,
    AUTHOR = {Elliott, Robert J. and van der Hoek, John},
     TITLE = {A general fractional white noise theory and applications to
              finance},
   JOURNAL = {Math. Finance},
  FJOURNAL = {Mathematical Finance. An International Journal of Mathematics,
              Statistics and Financial Economics},
    VOLUME = {13},
      YEAR = {2003},
    NUMBER = {2},
     PAGES = {301--330},
   
}

@article {Rogers1997,
    AUTHOR = {Rogers, L. C. G.},
     TITLE = {Arbitrage with fractional {B}rownian motion},
   JOURNAL = {Math. Finance},
  FJOURNAL = {Mathematical Finance. An International Journal of Mathematics,
              Statistics and Financial Economics},
    VOLUME = {7},
      YEAR = {1997},
    NUMBER = {1},
     PAGES = {95--105},
 
}

@book {Lord2014,
    AUTHOR = {Lord, Gabriel J. and Powell, Catherine E. and Shardlow, Tony},
     TITLE = {An introduction to computational stochastic {PDE}s},
    SERIES = {Cambridge Texts in Applied Mathematics},
 PUBLISHER = {Cambridge University Press, New York},
      YEAR = {2014},
 
}

@book {Hida1993,
    AUTHOR = {Hida, Takeyuki and Kuo, Hui-Hsiung and Potthoff, J\"{u}rgen and
              Streit, Ludwig},
     TITLE = {White noise},
    SERIES = {Mathematics and its Applications},
    VOLUME = {253},
      SUBTITLE = {An infinite-dimensional calculus},
 PUBLISHER = {Kluwer Academic Publishers Group, Dordrecht},
      YEAR = {1993},
 
}

@book {Thangavelu1993,
    AUTHOR = {Thangavelu, Sundaram},
     TITLE = {Lectures on {H}ermite and {L}aguerre expansions},
    SERIES = {Mathematical Notes},
    VOLUME = {42},
      NOTE = {With a preface by Robert S. Strichartz},
 PUBLISHER = {Princeton University Press, Princeton, NJ},
      YEAR = {1993},

}

@misc{Gjessing1993,
    AUTHOR = {Gjessing, H{\aa}kon K.},
    TITLE = {A note on the Wick product},
    note = {Statistical Report No. 23, University of Bergen},
    year = {1993},
}

@misc{Gjessing1993manu,
    AUTHOR = {Gjessing, H{\aa}kon K.},
    TITLE = {Wick calculus with applications to anticipating
stochastic differential equations},
    note = {Manuscript, University of Bergen},
    year = {1993},
}

@article {Benth2000,
    AUTHOR = {Benth, Fred Espen and Gjessing, H\aa kon K.},
     TITLE = {A nonlinear parabolic equation with noise. {A} reduction
              method},
   JOURNAL = {Potential Anal.},
  FJOURNAL = {Potential Analysis. An International Journal Devoted to the
              Interactions between Potential Theory, Probability Theory,
              Geometry and Functional Analysis},
    VOLUME = {12},
      YEAR = {2000},
    NUMBER = {4},
     PAGES = {385--401},
 
}

@article {Bender2003,
    AUTHOR = {Bender, Christian},
     TITLE = {An {I}t\^{o} formula for generalized functionals of a
              fractional {B}rownian motion with arbitrary {H}urst parameter},
   JOURNAL = {Stochastic Process. Appl.},
  FJOURNAL = {Stochastic Processes and their Applications},
    VOLUME = {104},
      YEAR = {2003},
    NUMBER = {1},
     PAGES = {81--106},

}

@article {Mishura2003,
    AUTHOR = {Mishura, Yu. S.},
     TITLE = {Quasilinear stochastic differential equations with a
              fractional-{B}rownian component},
   JOURNAL = {Teor. \u{I}movir. Mat. Stat.},
  FJOURNAL = {Teoriya \u{I}movirnoste\u{\i} ta Matematichna
              Statistika. Ki\"{\i}vs\cprime ki\u{\i} Universitet
              imeni Tarasa Shevchenka},
      YEAR = {2003},
    VOLUME = {68},
     PAGES = {95--106},

}

@article{Benth2003,
author = {Benth, Fred},
year = {2003},
month = {02},
pages = {303-324},
title = {On arbitrage-free pricing of weather derivatives based on fractional Brownian motion},
volume = {10},
journal = {Applied Mathematical Finance},

}

@article {Lin1995,
    AUTHOR = {Lin, S. J.},
     TITLE = {Stochastic analysis of fractional {B}rownian motions},
   JOURNAL = {Stochastics Stochastics Rep.},
  FJOURNAL = {Stochastics and Stochastics Reports},
    VOLUME = {55},
      YEAR = {1995},
    NUMBER = {1-2},
     PAGES = {121--140},
 
}

@article {PrakasaRao2016,
    AUTHOR = {Prakasa Rao, B. L. S.},
     TITLE = {Pricing geometric {A}sian power options under mixed fractional
              {B}rownian motion environment},
   JOURNAL = {Phys. A},
  FJOURNAL = {Physica A. Statistical Mechanics and its Applications},
    VOLUME = {446},
      YEAR = {2016},
     PAGES = {92--99},
    
}

@article {Garcin2022,
    AUTHOR = {Garcin, Matthieu},
     TITLE = {Forecasting with fractional {B}rownian motion: a financial
              perspective},
   JOURNAL = {Quant. Finance},
  FJOURNAL = {Quantitative Finance},
    VOLUME = {22},
      YEAR = {2022},
    NUMBER = {8},
     PAGES = {1495--1512},

}

@incollection {Nualart2003,
    AUTHOR = {Nualart, David},
     TITLE = {Stochastic integration with respect to fractional {B}rownian
              motion and applications},
 BOOKTITLE = {Stochastic models ({M}exico {C}ity, 2002)},
    SERIES = {Contemp. Math.},
    VOLUME = {336},
     PAGES = {3--39},
 PUBLISHER = {Amer. Math. Soc., Providence, RI},
      YEAR = {2003},

}

@article {Hu2003,
    AUTHOR = {Hu, Yaozhong and {\O}ksendal, Bernt},
     TITLE = {Fractional white noise calculus and applications to finance},
   JOURNAL = {Infin. Dimens. Anal. Quantum Probab. Relat. Top.},
  FJOURNAL = {Infinite Dimensional Analysis, Quantum Probability and Related
              Topics},
    VOLUME = {6},
      YEAR = {2003},
    NUMBER = {1},
     PAGES = {1--32},
    
}

@article {Coutin2002,
    AUTHOR = {Coutin, Laure and Qian, Zhongmin},
     TITLE = {Stochastic analysis, rough path analysis and fractional
              {B}rownian motions},
   JOURNAL = {Probab. Theory Related Fields},
  FJOURNAL = {Probability Theory and Related Fields},
    VOLUME = {122},
      YEAR = {2002},
    NUMBER = {1},
     PAGES = {108--140},
    
}

@article {Lyons1994,
    AUTHOR = {Lyons, Terry},
     TITLE = {Differential equations driven by rough signals. {I}. {A}n
              extension of an inequality of {L}. {C}. {Y}oung},
   JOURNAL = {Math. Res. Lett.},
  FJOURNAL = {Mathematical Research Letters},
    VOLUME = {1},
      YEAR = {1994},
    NUMBER = {4},
     PAGES = {451--464},
     
}

@incollection {Oksendal2007,
    AUTHOR = {{\O}ksendal, Bernt},
     TITLE = {Fractional {B}rownian motion in finance},
 BOOKTITLE = {Stochastic economic dynamics},
     PAGES = {11--56},
 PUBLISHER = {Cph. Bus. Sch. Press, Frederiksberg},
      YEAR = {2007},

}

@article {Bender2008,
    AUTHOR = {Bender, Christian and Sottinen, Tommi and Valkeila, Esko},
     TITLE = {Pricing by hedging and no-arbitrage beyond semimartingales},
   JOURNAL = {Finance Stoch.},
  FJOURNAL = {Finance and Stochastics},
    VOLUME = {12},
      YEAR = {2008},
    NUMBER = {4},
     PAGES = {441--468},
   
}

@article{Simonsen2003,
title = {Measuring anti-correlations in the nordic electricity spot market by wavelets},
journal = {Physica A: Statistical Mechanics and its Applications},
volume = {322},
pages = {597-606},
year = {2003},
author = {Ingve Simonsen}
}

@article {Brody2002,
    AUTHOR = {Brody, Dorje C. and Syroka, Joanna and Zervos, Mihail},
     TITLE = {Dynamical pricing of weather derivatives},
   JOURNAL = {Quant. Finance},
  FJOURNAL = {Quantitative Finance},
    VOLUME = {2},
      YEAR = {2002},
    NUMBER = {3},
     PAGES = {189--198},
  
}

@article {Akinlar2020,
    AUTHOR = {Akinlar, M. A. and Inc, Mustafa and G\'{o}mez-Aguilar, J. F.
              and Boutarfa, B.},
     TITLE = {Solutions of a disease model with fractional white noise},
   JOURNAL = {Chaos Solitons Fractals},
  FJOURNAL = {Chaos, Solitons \& Fractals},
    VOLUME = {137},
      YEAR = {2020},

}

@book {Benth2013,
    AUTHOR = {Benth, Fred Espen and \v{S}altyt\.{e} Benth, J\={u}rat\.{e}},
     TITLE = {Modeling and pricing in financial markets for weather
              derivatives},
    SERIES = {Advanced Series on Statistical Science \& Applied Probability},
    VOLUME = {17},
 PUBLISHER = {World Scientific Publishing Co. Pte. Ltd., Hackensack, NJ},
      YEAR = {2013},
    
}

@article{Rivero2016,
title = {A New Approach for Time Series Forecasting: Bayesian Enhanced by Fractional Brownian Motion with Application to Rainfall Series},
journal = {International Journal of Advanced Computer Science and Applications},
year = {2016},
publisher = {The Science and Information Organization},
volume = {7},
number = {3},
author = {Cristian Rodriguez Rivero and Daniel Patiño and Julian Pucheta and Victor Sauchelli}
}

@article{Lu2003,
author = {Lu, Silong and Molz, F. and Liu, Hui-Hai},
year = {2003},
pages = {15-25},
title = {An efficient, three-dimensional, anisotropic, fractional Brownian motion and truncated fractional L\'evy motion simulation algorithm based on successive random additions},
volume = {29},
journal = {Computers \& Geosciences},
}

@article {Jamshidi2021,
    AUTHOR = {Jamshidi, Nahid and Kamrani, Minoo},
     TITLE = {Convergence of a numerical scheme associated to stochastic
              differential equations with fractional {B}rownian motion},
   JOURNAL = {Appl. Numer. Math.},
  FJOURNAL = {Applied Numerical Mathematics. An IMACS Journal},
    VOLUME = {167},
      YEAR = {2021},
     PAGES = {108--118},

}

@article {Mishura2008article,
    AUTHOR = {Mishura, Yu. and Shevchenko, G.},
     TITLE = {The rate of convergence for {E}uler approximations of
              solutions of stochastic differential equations driven by
              fractional {B}rownian motion},
   JOURNAL = {Stochastics},
  FJOURNAL = {Stochastics. An International Journal of Probability and
              Stochastic Processes},
    VOLUME = {80},
      YEAR = {2008},
    NUMBER = {5},
     PAGES = {489--511},
  
}

@article {Liu2019,
    AUTHOR = {Liu, Yanghui and Tindel, Samy},
     TITLE = {First-order {E}uler scheme for {SDE}s driven by fractional
              {B}rownian motions: the rough case},
   JOURNAL = {Ann. Appl. Probab.},
  FJOURNAL = {The Annals of Applied Probability},
    VOLUME = {29},
      YEAR = {2019},
    NUMBER = {2},
     PAGES = {758--826},
     
}

@article {Zhang2021,
    AUTHOR = {Zhang, Shao-Qin and Yuan, Chenggui},
     TITLE = {Stochastic differential equations driven by fractional
              {B}rownian motion with locally {L}ipschitz drift and their
              implicit {E}uler approximation},
   JOURNAL = {Proc. Roy. Soc. Edinburgh Sect. A},
  FJOURNAL = {Proceedings of the Royal Society of Edinburgh. Section A.
              Mathematics},
    VOLUME = {151},
      YEAR = {2021},
    NUMBER = {4},
     PAGES = {1278--1304},
    
}

@misc{Leon2023,
      title={Euler scheme for SDEs driven by fractional Brownian motions: Malliavin differentiability and uniform upper-bound estimates}, 
      author={Jorge A. Le\'on and Yanghui Liu and Samy Tindel},
      year={2023},
      eprint={2305.10365},
      archivePrefix={arXiv},
  
}

@article {Clark1987,
    AUTHOR = {Clark, Dean S.},
     TITLE = {Short proof of a discrete {G}ronwall inequality},
   JOURNAL = {Discrete Appl. Math.},
  FJOURNAL = {Discrete Applied Mathematics. The Journal of Combinatorial
              Algorithms, Informatics and Computational Sciences},
    VOLUME = {16},
      YEAR = {1987},
    NUMBER = {3},
     PAGES = {279--281},
 
}

@article {Russo1993,
    AUTHOR = {Russo, Francesco and Vallois, Pierre},
     TITLE = {Forward, backward and symmetric stochastic integration},
   JOURNAL = {Probab. Theory Related Fields},
  FJOURNAL = {Probability Theory and Related Fields},
    VOLUME = {97},
      YEAR = {1993},
    NUMBER = {3},
     PAGES = {403--421},
  
}

@article {Cheridito2005,
    AUTHOR = {Cheridito, Patrick and Nualart, David},
     TITLE = {Stochastic integral of divergence type with respect to
              fractional {B}rownian motion with {H}urst parameter
              {$H\in(0,{\frac{1}{2}})$}},
   JOURNAL = {Ann. Inst. H. Poincar\'e{} Probab. Statist.},
  FJOURNAL = {Annales de l'Institut Henri Poincar\'e. Probabilit\'es et
              Statistiques},
    VOLUME = {41},
      YEAR = {2005},
    NUMBER = {6},
     PAGES = {1049--1081},
  
}

@article {Guasoni2021,
    AUTHOR = {Guasoni, Paolo and Mishura, Yuliya and R\'asonyi, Mikl\'os},
     TITLE = {High-frequency trading with fractional {B}rownian motion},
   JOURNAL = {Finance Stoch.},
  FJOURNAL = {Finance and Stochastics},
    VOLUME = {25},
      YEAR = {2021},
    NUMBER = {2},
     PAGES = {277--310},

}

@article {Czichowsky2017,
    AUTHOR = {Czichowsky, Christoph and Schachermayer, Walter},
     TITLE = {Portfolio optimisation beyond semimartingales: shadow prices
              and fractional {B}rownian motion},
   JOURNAL = {Ann. Appl. Probab.},
  FJOURNAL = {The Annals of Applied Probability},
    VOLUME = {27},
      YEAR = {2017},
    NUMBER = {3},
     PAGES = {1414--1451},
 
}

@article {Han2023,
    AUTHOR = {Han, Y. and Zheng, X.},
     TITLE = {Approximate pricing of derivatives under fractional stochastic
              volatility model},
   JOURNAL = {ANZIAM J.},
  FJOURNAL = {The ANZIAM Journal. The Australian \& New Zealand Industrial
              and Applied Mathematics Journal},
    VOLUME = {65},
      YEAR = {2023},
    NUMBER = {3},
     PAGES = {229--247},

}

\end{document}